\documentclass[12pt,reqno]{amsart}

\usepackage{amssymb}
\usepackage{amsmath}
\usepackage{mathpazo}
\usepackage{mathrsfs}
\usepackage{calc}
\usepackage{pstricks}

\theoremstyle{plain}
\newtheorem{theorem}{Theorem}[section]
\newtheorem{prop}[theorem]{Proposition}
\newtheorem{lma}[theorem]{Lemma}
\newtheorem{cor}[theorem]{Corollary}

\theoremstyle{definition}
\newtheorem{defn}[theorem]{Definition}
\newtheorem{ex}[theorem]{Example}
\newtheorem{rmk}[theorem]{Remark}

\newrgbcolor{medgray}{0.5 0.5 0.5}
\newrgbcolor{liggray}{0.75 0.75 0.75}

\newcommand{\N}{\mathbb{N}}
\newcommand{\Z}{\mathbb{Z}}
\newcommand{\gL}{\mathcal{L}}
\newcommand{\gR}{\mathcal{R}}
\newcommand{\gH}{\mathcal{H}}
\newcommand{\gD}{\mathcal{D}}
\newcommand{\gJ}{\mathcal{J}}
\newcommand{\K}{\mathcal{K}}
\newcommand{\Y}{\mathcal{Y}}
\newcommand{\B}{\mathcal{B}}
\newcommand{\A}{\mathcal{A}}
\newcommand{\VV}{\mathcal{V}}
\newcommand{\C}{\mathbf{C}}
\newcommand{\V}{\mathbf{V}}
\newcommand{\W}{\mathbf{W}}
\newcommand{\CR}{\mathbf{CR}}
\newcommand{\I}{\mathbf{I}}
\newcommand{\Sast}{\mathbf{S^{\ast}}}
\newcommand{\SI}{\mathbf{SI}}
\newcommand{\T}{\mathbf{T}}
\newcommand{\G}{\mathbf{G}}
\newcommand{\lat}{\mathscr{L}}
\newcommand{\lfin}{{\bf L}_{\bf F}}
\newcommand{\rfin}{{\bf R}_{\bf F}}
\newcommand{\hfin}{{\bf H}_{\bf F}}
\newcommand{\dfin}{{\bf D}_{\bf F}}
\newcommand{\jfin}{{\bf J}_{\bf F}}
\newcommand{\kfin}{{\bf K}_{\bf F}}

\newcommand{\gen}[1]{\left\langle #1\right\rangle}
\newcommand{\mapright}[1]{\stackrel{#1}{\longrightarrow}}
\newcommand{\mapleft}[1]{\stackrel{#1}{\longleftarrow}}
\newcommand{\inv}{^{-1}}

\newcommand{\dom}{\text{dom }}
\newcommand{\im}{\text{im }}
\newcommand{\id}{\text{id }\!}

\begin{document}

\setlength{\parindent}{0pt} \thispagestyle{empty}


\begin{flushleft}

\Large

\textbf{Local finiteness for Green relations in ($I$-)semigroup varieties}

\end{flushleft}


\normalsize

\vspace{-2.25mm}

\rule{\textwidth}{1.5pt}

\vspace{2mm}


\begin{flushright}

\large
\textsc{Pedro V. Silva}

\vspace{3pt}
\scriptsize
\textit{Centro de Matem\'{a}tica, Faculdade de Ci\^{e}ncias, Universidade do Porto,} \\
\textit{Rua Campo Alegre 687, 4169-007 Porto, Portugal} \\
\texttt{pvsilva@fc.up.pt}

\vspace{2mm}

\large
\textsc{Filipa Soares}

\vspace{3pt}
\scriptsize
\textit{\'Area Departamental de Matem\'atica,} \\
\textit{ISEL --- Instituto Superior de Engenharia de Lisboa, Instituto Polit\'ecnico de Lisboa,} \\
\textit{Rua Conselheiro Em\'idio Navarro 1, 1959-007 Lisboa, Portugal} \\
\& \hspace{2mm} \textit{CEMAT--CI\^ENCIAS,} \\
\textit{Dep. Matem\'atica, Faculdade de Ci\^encias, Universidade de Lisboa} \\
\textit{Campo Grande, Edif\'icio C6, Piso 2, 1749-016 Lisboa, Portugal} \\
\texttt{falmeida@adm.isel.pt}

\vspace{5mm}

\normalsize

\today

{\footnotesize Version: \texttt{green\_varieties\_10.pdf}}

\end{flushright}

\medskip

\begin{flushleft}

\small

{\bf Abstract.}
In this work, the lattice of varieties of semigroups
and the lattice of varieties of $I$-semigroups
(a common setting for both the variety of completely regular semigroups
and the variety of inverse semigroups)
are studied with respect to the following concepts:
a variety $\V$ of ($I$-)semigroups is said to be locally $\K$-finite,
where $\K$ stands for any of the five Green's relations,
if every finitely generated semigroup from $\V$ has only finitely many
(distinct) $\K$-classes.

\medskip


\medskip


\medskip

\date{\today}

\end{flushleft}

\bigskip

\section{Introduction} \label{sec:intro}

Varieties of semigroups have been a theme of research so widely and diversely
pursued that any listing of developements would fall inefficiently short
of the actual achievements.
Among some other special types, locally finite varieties of semigroups
play a major role in the study of varieties of semigroups,
with connections with many important problems,
both solved and open;
see, for example, \cite{Sap14}.

In this work, we consider a certain kind of generalizations
of the concept of locally finite variety that concerns one of the
key instruments in the study of semigroups --- the Green's relations ---
which we have termed $\K$-finite variety,
$\K$ being one of the five Green's relations.
Namely, we say that a variety $\V$ is \emph{locally $\K$-finite}
if each finitely generated semigroup belonging to $\V$
has but finitely many $\K$-classes.
Thus being, the problem of characterizing varieties with respect to these notions
intersects the Burnside Problem for semigroups,
which, in its standard formulation,
asks which varieties of semigroups contain infinite periodic
finitely generated semigroups.
As one might expect, it gives rise to some nontrivial questions.
This introdutory work establishes the basic properties of $\K$-finite varieties
and the classification of some of the most
relevant varieties of semigroups (more precisely, varieties of semigroups,
varieties of completely regular semigroups and varieties of inverse semigroups)
with respect to these notions.

The paper is organized as follows:
section~\ref{sec:bg} provides the necessary definitions and facts
on semigroups and varieties;
section~\ref{sec:classes} presents the notion of $\K$-finite semigroup,
for $\K\in \{\gH,\gL,\gR,\gD,\gJ\}$,
investigates the hierarchy of these notions,
and describes its behaviour with respect to the basic variety operators,
providing a number of examples;
finally, section~\ref{sec:vars} presents the definition of $\K$-finite variety,
establishes some of its properties,
and deals with the classification of the
varieties of completely regular semigroups, varieties of inverse semigroups,
varieties of $I$-strict semigroups, and varieties of semigroups
(subsections~\ref{subsec:CR}, \ref{subsec:INV}, \ref{subsec:SIS},
and \ref{subsec:SMG}, respectively).

\section{Background} \label{sec:bg}

A \emph{variety} is a class $\V$ of algebras, all of the same signature,
which is closed under taking subalgebras, homomorphic images, and direct products,
or, equivalently, that consists of all algebras that satisfy
a certain identity $u=v$ or set of identities
$\{u_{\lambda}=v_{\lambda}\}_{\lambda\in \Lambda}$
(written $\V=[u=v]$ and $\V=[u_{\lambda}=v_{\lambda}]_{\lambda\in \Lambda}$,
respectively).

If $\C$ is a class of algebras,
$\gen{\C}$ denotes the \emph{variety generated by} $\C$,
that is, the smallest variety that contains $\C$.
As usual, for $\C=\{A\}$, we write simply $\gen{A}$.
It is well known that

\begin{theorem} \label{T:var-characterization}
 Let $\C$ be a class of algebras and $A$ an algebra of the same signature
 that the algebras in $\C$.
 Then $A\in \gen{\C}$ if and only if there exist $C_{\lambda}\in \C$
 ($\lambda \in \Lambda$), a subalgebra $B\leq \Pi_{\lambda \in \Lambda} C_{\lambda}$,
 and an epimorphism $\psi\colon B \to A$.
\end{theorem}

%
%

If $\V$ is a variety, then the set $\mathscr{L}(\V)$ of all subvarieties of $\V$
is a complete lattice under inclusion, with
$$\bigwedge_{\lambda\in \Lambda} \V_{\lambda}
  = \bigcap_{\lambda\in \Lambda} \V_{\lambda}$$
and
\begin{align*}
 \bigvee_{\lambda\in \Lambda} \V_{\lambda}
  &= \bigcap \; \, \{\W \subseteq \V \colon \V_{\lambda} \subseteq \W \text{ for all } \lambda\in \Lambda \} \\
  &= \{A \in \V \colon \text{there exist } \Gamma \subseteq \Lambda, \,
      V_{\gamma}\in \V_{\gamma} (\gamma \in \Gamma), \, B\leq \Pi_{\gamma \in \Gamma} V_{\gamma}, \\
  &  \qquad \qquad \quad \; \, \text{and an epimorphism } \psi\colon B \to A\} \, .
\end{align*}

\medskip

When dealing with semigroups in general, these are viewed as algebras with
a single (binary and associative) operation,
but some specific, yet important, classes of semigroups
are best considered as algebras of type
$(2,1)$, that is, semigroups endowed with a unary operation $a\mapsto a'$,
and amongst them is the \emph{variety of $I$-semigroups},
which is defined by the identities
\begin{equation} \label{Eq:varIsmgs}
 x(yz)=(xy)z \, , \; (x')' =x \, , \; xx'x=x
\end{equation}
(see \cite{How95}).
Within the variety of $I$-semigroups one finds, for example:
the subvariety \emph{of completely regular semigroups},
where the unary operation maps each element to the only inverse of $a$ with which $a$
commutes and is thus defined, in addition to (\ref{Eq:varIsmgs}),
by the identity
$$xx'=x'x \, ;$$
the subvariety \emph{of inverse semigroups},
where the unary operation maps each element to its (unique) inverse,
and which is defined, together with (\ref{Eq:varIsmgs}), by the identity
$$xx'yy'=yy'xx'$$
or, equivalently, by the identities
$$(xy)'=y'x' \, , \; xx'x'x=x'xxx' \, ;$$
and subvariety the \emph{strict $I$-semigroups},
considered by Petrich and Reilly in \cite{PetRei84} and
which is defined, in addition to (\ref{Eq:varIsmgs}), by the identities:
$$xx'x'x=x'xxx' \, , \; (xyx')(xyx')'=(xyx')'(xyx') \, ,$$
$$x(yz)'w=xz'y'w \, , \;  (xy)'=(x'xy)'(xyy')' \, .$$
These three varieties will be denoted by $\CR$, $\I$ and $\SI$,
respectively.\footnote{In \cite{PetRei84}, Petrich and Reilly have termed
strict $I$-semigroups by ``strict $^\ast$-semigroups''
and denoted the variety which includes them all by $\Sast$;
Howie~\cite{How95}, however, calls \emph{$\ast$-semigroup}
(or \emph{semigroup with involution})
to a unary semigroup obliged only to satisfy the identity $(xy)'=y'x'$,
which, as noted in \cite{PetRei84},
does not follow from the identities defining $\SI$.}
As usual, we will write $a\inv$ instead of $a'$ in the context of inverse semigroups
and denote by $a^0$ the (unique) idempotent in $H_a$
in the context of completely regular semigroups.

\medskip

Given a semigroup $S$, the Green's relations on $S$ are
(the equivalence relations) defined by: for all $a,b\in S$,
$$\begin{array}{rcrcl}
   a \gL b & \iff & S^1a \!	&= \!	&S^1b \\
   a \gR b & \iff & aS^1 \!	&= \!	&bS^1 \\
   a \gJ b & \iff & S^1aS^1 \!	&= \!	&S^1bS^1 \, ,
  \end{array}$$
along with
$$\gH = \gL \cap \gR \; \text{ and } \; \gD = \gL \vee \gR \, .$$
Recall that, since $\gL$ and $\gR$ always commute, we actually have
$$\gD = \gL \circ \gR = \gR \circ \gL$$
and that, within the case of regular semigroups, 
the superscript ``$1$'' in the definition of $\gL$, $\gR$, and $\gJ$ can be dropped.
Also recall that a semigroup is said to be \emph{combinatorial}
(respectively, \emph{cryptic}) if $\gH=1_S$ (respectively, if $\gH$ is a congruence).

An element $a$ in a semigroup $S$ has \emph{finite order}
if the monogenic subsemigroup $\gen{a}$ is finite
and has \emph{infinite order} otherwise.
A semigroup $S$ is \emph{periodic} if every element has finite order.
Note that, in case $S$ is an inverse semigroup,
it does not matter whether one considers the subsemigroup of $S$
generated by $a$ or the inverse subsemigroup generated by $a$:

\begin{lma}
 Let $S$ be an inverse semigroup.
 Then every monogenic subsemigroup of $S$ is finite
 if and only if every monogenic inverse subsemigroup of $S$ is finite.
\end{lma}

\begin{proof}
 Suppose every monogenic subsemigroup of $S$ is finite and
 let $a\in S$ and $T$ be the monogenic inverse subsemigroup of $S$
 generated be $a$.
 Then every element in $T$ is of the form $a^la^{-m}a^n$
 for some nonnegative integers $l$, $m$ and $n$ not all zero.
 Since both the monogenic subsemigroup of $S$ generated by $a$
 and the monogenic subsemigroup of $S$ generated by $a\inv$ are finite
 by assumption, it follows that $T$ has only finitely many elements.
 As for the converse, we have that, if $a\in S$,
 then the monogenic subsemigroup generated by $a$
 embeds in the monogenic inverse subsemigroup generated by $a$,
 hence the finiteness of the second implies the finiteness of the first.
\end{proof}

%

A semigroup $S$ is an \emph{epigroup} (or \emph{group-bound},
according to Howie in \cite{How95}), if, for every $a\in S$,
there exists a positive integer $m$ such that $a^m$ belongs to a subgroup of $S$.
In case $S$ is inverse, this is the same as existing a positive integer $m$
such that $a^ma^{-m}=a^{-m}a^m$.
We have

\begin{prop}[\cite{CliPre67}, Theorem~6.45] \label{P:D=J-epigroup}
 If $S$ is an epigroup, then $\gD=\gJ$.
\end{prop}

which includes, of course, the very well known case of periodic semigroups.

%
%
%
%

\medskip

A variety $\V$ of ($I$-)semigroups is \emph{combinatorial}
(respectively, \emph{cryptic}, \emph{periodic})
if it consists solely of combinatorial (respectively, cryptic, periodic)
($I$-)semigroups.

\medskip

For a subsemigroup $T$ of a semigroup $S$, we will denote, as usual,
the Green relation $\K$ in  $T$ by $\K^T$
(and, when necessary, the Green relation $\K$ in $S$ by $\K^S$);
thus, for example,
$$(a,b)\in \gL^S \iff S^1a=S^1b \quad \text{ while } \quad (a,b)\in \gL^T \iff T^1a=T^1b \, .$$
Clearly, $\gL^T \subseteq \gL^S \cap (T\times T)$
and similarly for the remaining Green's relations
--- inclusions that are strict in general.
However,
\begin{prop}[\cite{How95}, Proposition~2.4.2] \label{P:reg-subsmg}
 If $T$ is a regular subsemigroup of $S$, then
 $$ \gL^T=\gL^S \cap (T\times T) \, , \; \gR^T=\gR^S \cap (T\times T) \, \text{ and } \;
	\gH^T=\gH^S \cap (T\times T) \, .$$
\end{prop}
a property neither $\gD$ nor $\gJ$ need follow in general.


%
%


Another important property of completely regular semigroups concerning Green's relations is

\begin{theorem}[\cite{PetRei99}, Theorem~II.4.5] \label{T:D=JinCR}
 In any completely regular semigroup, $\gD=\gJ$.
\end{theorem}


%
%


A \emph{band} is a semigroup in which every element is an idempotent.

A band $B$ is \emph{regular} if $axya=axaya$, for all $a,x,y\in B$,
and \emph{left regular} if $ax=axa$, for all $a,x\in B$.
\emph{Right regular bands} are defined dually.
A band $B$ is \emph{normal} if $axya=ayxa$, for all $a,x,y\in B$.

%

\medskip

A semigroup is a \emph{cryptogroup} if it is cryptic and completely regular.

A semigroup $S$ is \emph{orthodox} if it is regular and its idempotents form a subsemigroup.
An orthodox completely regular semigroup is called an \emph{orthogroup}
and an orthodox cryptogroup (that is, a completely regular semigroup
which is both cryptic and orthodox) is called an \emph{orthocryptogroup}.

If $S$ is an orthogroup and $E(S)$ is a (left, right) regular band,
we say that $S$ is a \emph{(left, right) regular orthogroup}.
Similarly, \emph{regular orthocryptogroups} are 
orthocryptogroups with a regular band of idempotents
and \emph{normal orthocryptogroups} are 
orthocryptogroups with a normal band of idempotents.


\medskip

A semigroup $S$ is called a \emph{band of semigroups} $S_{\alpha}$,
with $\alpha \in \Y$, if $S$ is a disjont union of the semigroups $S_{\alpha}$
and the corresponding partition is a congruence.
Thus being, the quotient semigroup of $S$ over this congruence, $\Y$, is a band.
If this band is in addition commutative,
we say that $S$ is a \emph{semilattice of semigroups}.
In either case, we write $S=(Y;S_{\alpha})$.
If all semigroups are of some special type,
we say that $S$ is a band (respectively, semilattice)
of semigroups of that particular type.
For example, it can be shown that \emph{Clifford semigroups},
that is, completely regular semigroups in which any idempotent
commutes with any element, are semilattices of groups.

For cryptogroups, we have:

\begin{theorem}
[\cite{PetRei99}, Theorem~II.8.1]
 The following conditions on a semigroup $S$ are equivalent:
 \begin{itemize}
  \item[(i)] $S$ is a cryptogroup
  \item[(ii)] $S$ is a band of groups.
 \end{itemize}
\end{theorem}

Thus, cryptogroups can be represented in the form $S=(\B;S_{\beta})$,
where $\B$ is the band $S/\gH$ and each $S_{\beta}$ is an $\gH$-class of $S$ and a group.

It is well known that all the above mentioned classes of completely regular semigroups
are subvarieties of $\CR$;
for them, we will use the following notation:
$$\begin{array}{ll}
   \mathbf{OG}		&\text{orthogroups} \\
   \mathbf{RO}		&\text{regular orthogroups} \\
   \mathbf{ROL^*}	&\text{regular orthogroups in which $\gL$ is a congruence} \\
   \mathbf{ROR^*}	&\text{regular orthogroups in which $\gR$ is a congruence} \\
   \mathbf{LRO}		&\text{left regular orthogroups} \\
   \mathbf{RRO}		&\text{right regular orthogroups}
  \end{array}
$$

$$\begin{array}{ll}
   \mathbf{BG}		&\text{cryptogroups} \\
   \mathbf{OBG}		&\text{orthocryptogroups} \\
   \mathbf{ROBG}	&\text{regular orthocryptogroups} \\
   \mathbf{ONBG}	&\text{normal orthocryptogroups} \\
   \mathbf{SG}		&\text{Clifford semigroups}.
  \end{array}
$$

We say that a semigroup $L$ in which $xy=x$ for all $x,y \in L$
is a \emph{left zero semigroup}
and that a completely regular semigroup $S$ in which $\gL=S\times S$
is a \emph{left group}.
In the sequel, the following will be of use:

\begin{lma}[\cite{PetRei99}, Proposition~III.5.9] \label{L:LG}
 The following are equivalent on a completely regular semigroup:
 \begin{itemize}
  \item[(i)] $S$ is a left group;
  \item[(ii)] $S$ is isomorphic to the direct product of a left zero semigroup and a group.
 \end{itemize}
\end{lma}

\begin{lma}[\cite{PetRei99}, Lemma~V.3.1] \label{L:ROL}
 The following are equivalent on a completely regular semigroup:
 \begin{itemize}
  \item[(i)] $S \in \mathbf{LRO}$;
  \item[(ii)] $S$ is a semilattice of left groups.
 \end{itemize}
\end{lma}

\begin{theorem}[\cite{PetRei99}, Theorem~V.6.2] \label{T:ROLstar}
 The following are equivalent on a completely regular semigroup:
 \begin{itemize}
  \item[(i)] $S \in \mathbf{ROL^*}$;
  \item[(ii)] $S$ embeds in a direct product of a left regular orthogroup
   and a right regular band;
  \item[(iii)] $S$ satisfies the identity $x(y^0z)^0x=xy^0x^0z^0x$.
 \end{itemize}
\end{theorem}

As noted in \cite{PetRei99},
it follows from Lemma~\ref{L:ROL} that $S\in \mathbf{LRO}$
if and only if $\gL=\gD$ on $S$
and from Theorem~\ref{T:ROLstar} that $\mathbf{ROL^*} = \mathbf{LRO}\vee \mathbf{ROBG}$.

\medskip

%

\medskip

We now recall some notions about free inverse semigroups
and Sch\"{u}tzenberger graphs.

Let $A$ be a nonempty alphabet and let $A\inv$ be the set of formal inverses of $A$.
If $a\in A\cup A\inv$, we say that the edge $q \mapright{a\inv} p$
is the inverse of the edge $p \mapright{a} q$.
An automaton is said to be \emph{inverse} if
it is deterministic, has a single initial vertex and a single final vertex,
its underlying graph is connected,
and its edge set is closed under inversion.

Given $u=a_1\ldots a_n$, with each $a_i \in A\cup A\inv$,
the \emph{linear automaton $L(u)$ of $u$} is defined as the inverse automaton
\[
 \parbox{4cm}{
  \begin{pspicture}(0,0)(4,1)
   \psdot(0,.4)
   \psdot(1,.4)
   \psdot(2,.4)
   \psdot(3,.4)
   \psdot(4,.4)
   \psline[linewidth=.4pt,arrows=->](-.3,.8)(-.05,.45)
   \psline[linewidth=.4pt,arrows=->](4,.4)(4.3,.8)
   \psline[linewidth=.4pt,arrows=->](0,.4)(.925,.4)
   \psline[linewidth=.4pt,arrows=->](1,.4)(1.925,.4)
   \psline[linewidth=.4pt,arrows=->](3,.4)(3.925,.4)
   \uput[u](.5,.35){\small $a_1$}
   \uput[u](1.5,.35){\small $a_2$}
   \rput(2.55,.4){\small $\cdots$}
   \uput[u](3.5,.35){\small $a_n$}
 \end{pspicture}}
\]
(recall that the inverse edges are not depicted)
and the \emph{Munn tree $MT(u)$} is the deterministic automaton obtained from $L(u)$
by sucessively folding all pairs of edges of the form
$q \mapleft{a} p \mapright{a} r$ with $a\in A\cup A\inv$ (see~\cite{Mun74}).
As is well-known, the process is confluent and the
\emph{free inverse semigroup on $A$}
can be viewed as the quotient $FIS_A=(A\cup A\inv)^+/\rho$,
where $\rho$ is the congruence on $(A\cup A\inv)^+$ defined by
$$u\rho v \Longleftrightarrow MT(u) \approx MT(v) \, ,$$
for all $u,v \in (A\cup A\inv)^+$.
Given $u \in (A\cup A\inv)^+$, we denote by $\bar{u}$ the reduced form of $u$,
obtained by deleting all factors of the form $aa\inv$, with $a\in A\cup A\inv$;
then,
$$u\rho \in E(FIS_A) \Longleftrightarrow \bar{u}=1$$
holds for every $u \in (A\cup A\inv)^+$.

Let ${\rm InvS} \gen{A \mid R}$ be an inverse semigroup presentation;
thus, $R\subseteq (A\cup A\inv)^+ \times (A\cup A\inv)^+$
and the inverse semigroup defined by the presentation is the quotient
$(A\cup A\inv)^+/\tau$, where $\tau$ is the congruence generated by $\rho\cup R$.
Given $u \in (A\cup A\inv)^+$,
the \emph{Sch\"{u}tzenberger automaton $\A(u)$ of $u$}
has vertex set $R_{u\tau}$,
edges $v\tau \mapright{a} (va)\tau$
(whenever $a\in A\cup A\inv$ and $v\tau,(va)\tau \in R_{u\tau}$, evidently),
initial vertex $(uu\inv)\tau$, and final vertex $u\tau$.
As shown in \cite{JBS90}, these are inverse automata and provide a solution
for the word problem through the equivalence
$$u\tau v \Longleftrightarrow \A(u) \approx \A(v) \, .$$
The underlying graph of the Sch\"{u}tzenberger automaton $\A(u)$
is called the \linebreak \emph{Sch\"{u}tzenberger graph} of $u$, and denoted $S\Gamma(u)$.
These can be very useful when dealing with Green's relations;
for example, the set of $\gD$-classes of the semigroup $(A\cup A\inv)^+/\tau$
can be identitied with the set of isomorphism classes of the
Sch\"{u}tzenberger graphs of ${\rm InvS} \gen{A \mid R}$.

In general, Sch\"{u}tzenberger graphs are not effectively constructible,
but an inductive procedure has been designed by Stephen to approximate them.
For \linebreak $u\in (A\cup A\inv)^+$,
define the \emph{Sch\"{u}tzenberger sequence $(S\Gamma_n(u))_n$} as follows.
Starting with $S\Gamma_1(u)=MT(u)$, and provided $S\Gamma_n(u)$ is defined,
consider the automaton $S\Gamma'_n(u)$ obtained from $S\Gamma_n(u)$
by performing all possible \emph{$R$-expansions}, that is:
if $(r,s)\in R\cup R\inv$ and there exists a path $p\mapright{r}q$ in $S\Gamma_n(u)$,
but no path $p\mapright{s}q$, adjoin one such path.
Then, take $S\Gamma_{n+1}(u)$ to be the automaton obtained from $S\Gamma'_n(u)$
by identitying all pairs of edges $q \mapleft{a} p \mapright{a} r$
with $a\in A\cup A\inv$.
Although it needn't be always the case,
in many circumstances each automaton $S\Gamma_n(u)$ can effectively constructed
(for example, when $R$ is finite).
Using an appropriate quasi-order in the class of inverse automata,
the Sch\"{u}tzenberger automaton can be viewed as the direct limit
of the Sch\"{u}tzenberger sequence.

With respect to inverse semigroups, we will denote by:
\begin{itemize}
 \item $B=\N_0\times \N_0$ the bicyclic monoid, whose operation is defined by
  $$(m,n)(p,q)=(m-n+n\vee p,q-p+n\vee p)\,,$$
  for all $(m,n),(p,q)\in B$
  (where $n\vee p=\max\{n,p\}$);
 \item $B_2$ the five element Brandt semigroup,
  which can be defined by the inverse semigroup presentation
  $B_2 = {\rm InvS} \gen{a \mid a^2=0} \, ;$
 \item $FIS_a$ the free monogenic inverse semigroup;
 \item $M_n$, for $n$ a positive integer, the Rees quotient $FIS_a/I_n$
  where $I_n$ is the ideal of $FIS_a$ generated be $a^n$.
\end{itemize}

Concerning subvarieties of $\I$, the following notation will be used:
$$\begin{array}{ll}
   \T				&\text{trivial semigroups} \\
   \G				&\text{groups} \\
   \mathbf{SL}			&\text{semilattices} \\
   \mathbf{C_m}			&[x^m=x^{m+1}], \text{ for any positive integer $m$} \\
   \left[x^n\in G\right]	&[x^nx^{-n}=x^{-n}x^n], \text{ for any positive integer $n$.}
  \end{array}
$$

It is well known that the lattice of subvarieties of $\I$
contains an ideal with three isomorphic layers:
(i) $\V$, (ii) $\V\vee \mathbf{SL}$, and (iii) $\V\vee \gen{B_2}$,
where $\V$ runs over all varieties of groups (see Figure~\ref{F:Iclassification}).
Semigroups in these varieties are often called \emph{strict inverse semigroups}
(which should not be confounded with strict $I$-semigroups!).
We also recall

\begin{theorem}[\cite{Pet84}, Theorem~XII.4.11] \label{T:FISa=B}
 $\gen{FIS_a}=\gen{B}$.
\end{theorem}

and

\begin{cor}[\cite{PetRei84}, Corollary~5.5] \label{C:SI-join}
 $\SI=\mathbf{ONBG}\vee \gen{B_2}$.
\end{cor}




For undefined terms and further details, the reader is referred to
\cite{Pet84}, \cite{PetRei99}, and \cite{How95}.

\section{The classes $\hfin$, $\lfin$, $\rfin$, $\dfin$, and $\jfin$} \label{sec:classes}

Given a semigroup $S$ and $\K\in \{\gH,\gL,\gR,\gD,\gJ\}$,
we say that $S$ is \emph{$\K$-finite} if $S$ has only finitely many
(distinct) $\K$-classes.
Each of these classes contains a few well known families of semigroups.
For example, every group is $\gH$-finite,
every left (respectively, right) simple semigroup is $\gL$-finite
(respectively, $\gR$-finite),
every bisimple semigroup is $\gD$-finite,
and every simple semigroup is $\gJ$-finite.

Let $\kfin$ denote the class consisting of all $\K$-finite semigroups,
for each $\K\in \{\gH,\gL,\gR,\gD,\gJ\}$.

It is straightforward to see that $\{\hfin,\lfin,\rfin,\dfin,\jfin\}$
is the following poset with respect to class inclusion:
$$\begin{array}{ccccc}
   		&		& \jfin	&		& \\
   		&		& |	&		& \\
   		&		& \dfin	&		& \\
   		& \diagup	&	& \diagdown	& \\
   \lfin		&		&	&		& \rfin \\
		& \diagdown	&	& \diagup	& \\
		&		& \hfin	&		&
  \end{array}
$$

\begin{rmk} \label{R:class-poset}
 Notice that indeed $\hfin = \lfin \cap \rfin$.
\end{rmk}

These inclusions are all strict, as shown by the following examples.

\begin{ex} \label{E:bicyclic}
 The bicyclic monoid, $B$, is a bisimple semigroup;
 for that reason, it is both $\gD$- and $\gJ$-finite.
 However, $B$ is neither $\gL$- nor $\gR$-finite (nor, either way, $\gH$-finite).
\end{ex}

Since $(m,n)\gL (p,q)$ if and only if $n=q$ and $(m,n)\gR (p,q)$ if and only if $m=p$,
we have $|B/\gL|=\aleph_0$ and $|B/\gR|=\aleph_0$,
so that $B\notin \lfin$ and $B\notin \rfin$.

\begin{ex} \label{E:a^nb}
 Consider the inverse monoid defined by the inverse monoid presentation
 \[
  M =  {\rm InvM} \gen{a,b \mid b^2=b, \, b=baba\inv, \, aa\inv=1} \, .
 \]
 Then $M$ is $\gJ$-finite, while not $\gD$-finite.
\end{ex}

In view of the defining relations of $M$,
there are only two kinds of Sch\"{u}tzenberger graphs for the elements of $M$,
namely
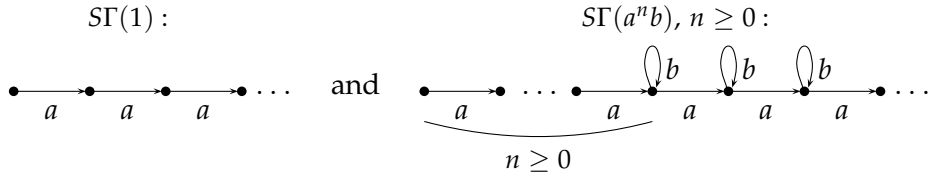
\[
 \parbox{12cm}{
  \begin{pspicture}(0,-1)(12,2)
   \psdot(0,.4)
   \psdot(1,.4)
   \psdot(2,.4)
   \psdot(3,.4)
   \psline[linewidth=.4pt,arrows=->](0,.4)(.95,.4)
   \psline[linewidth=.4pt,arrows=->](1,.4)(1.95,.4)
   \psline[linewidth=.4pt,arrows=->](2,.4)(2.95,.4)
   \uput[d](.5,.4){\small $a$}
   \uput[d](1.5,.4){\small $a$}
   \uput[d](2.5,.4){\small $a$}
   \uput[r](3,.375){\small $\cdots$}
   \uput[u](1.5,1){\footnotesize $S\Gamma(1)$~:}
   \uput[r](4,.5){\small and}
   \psdot(5.4,.4)
   \psdot(6.4,.4)
   \psdot(7.4,.4)
   \psdot(8.4,.4)
   \psdot(9.4,.4)
   \psdot(10.4,.4)
   \psdot(11.4,.4)
   \psline[linewidth=.4pt,arrows=->](5.4,.4)(6.35,.4)
   \uput[r](6.5,.375){\small $\cdots$}
   \psline[linewidth=.4pt,arrows=->](7.4,.4)(8.35,.4)
   \uput[d](5.9,.4){\small $a$}
   \uput[d](7.9,.4){\small $a$}
   \uput[d](8.9,.4){\small $a$}
   \uput[d](9.9,.4){\small $a$}
   \uput[d](10.9,.4){\small $a$}
   \psline[linewidth=.4pt,arrows=->](8.4,.4)(9.35,.4)
   \psline[linewidth=.4pt,arrows=->](9.4,.4)(10.35,.4)
   \psline[linewidth=.4pt,arrows=->](10.4,.4)(11.35,.4)
   \uput[r](11.4,.375){\small $\cdots$}
   \pscurve[linewidth=.4pt,arrows=->](8.4,.4)(8.3,.8)(8.4,1)(8.5,.8)(8.45,.45)
   \pscurve[linewidth=.4pt,arrows=->](9.4,.4)(9.3,.8)(9.4,1)(9.5,.8)(9.45,.45)
   \pscurve[linewidth=.4pt,arrows=->](10.4,.4)(10.3,.8)(10.4,1)(10.5,.8)(10.45,.45)
   \uput[r](8.4,.7){\small $b$}
   \uput[r](9.4,.7){\small $b$}
   \uput[r](10.4,.7){\small $b$}
   \parabola[linewidth=.4pt](5.4,0)(6.9,-.2)
   \uput[d](6.9,-.2){\footnotesize $n\ge 0$}
   \uput[u](8.7,1){\footnotesize $S\Gamma(a^nb), \, n\ge 0$~:}
 \end{pspicture}}
\]
according to whether there is no occurence of $b$ or there is at least one occurence of $b$,
respectively.
Since there is no map between $S\Gamma(a^nb)$ and $S\Gamma(a^pb)$, with $n\neq p$,
that preserves labeling, incidence, and orientation and is bijective on the vertices
and surjective on the edges, we conclude that, in this case,
$a^nb$ and $a^pb$ are not $\gD$-related (cf.~\cite[Theorem~3.4]{JBS90}).
Thus, $M$ has infinitely many $\gD$-classes.
However, $a^nb \gJ a^{n+1}b$ for any non-negative integer $n$,
as $a^{n+1}b=a\,a^nb \in Ma^nb$ and
\[
 a^nb = a^nb\inv = a^n(baba\inv)\inv = a^n a b\inv a\inv b\inv
 = a^{n+1} b \,a\inv b \in a^{n+1}bM \, ,
\]
so that $a^nb \gJ a^pb$ for any non-negative integers $n$ and $p$.
Therefore, $M\in \jfin\smallsetminus \dfin$. 

\medskip

As mentioned earlier, a left zero semigroup is trivially $\gL$-finite,
but it needs not be $\gR$-finite, evidently.
Thus, $\lfin\smallsetminus \rfin$, and hence $\lfin\smallsetminus \hfin$,
are nonempty,
and likewise so are $\rfin\smallsetminus \lfin$ and $\rfin\smallsetminus \hfin$.
Next, we provide another example of a semigroup in $\lfin\smallsetminus \rfin$,
which will be of use later on.

\begin{ex} \label{E:Pedro1}
 Let $\Z$ be endowed with the binary operation defined by
 $$m \circ n = \left\{
  \begin{array}{ll}
   m+n&\mbox{ if $m$ is even}\\
   m&\mbox{ if $m$ is odd.}
  \end{array}
 \right.$$
 Then $P = (\Z, \circ)$ has finitely many $\gL$-classes but
 infinitely many $\gR$-classes.
\end{ex}

Indeed, for all $m,n,k \in \Z$, we have
$$\begin{array}{l}
(2m+1)\circ (n \circ k) = 2m+1 = ((2m+1)\circ n) \circ k,\\
(2m)\circ ((2n+1) \circ k) = 2m+2n+1 = ((2m)\circ (2n+1)) \circ k,\\
(2m)\circ ((2n) \circ k) = 2m+2n+k = ((2m)\circ (2n)) \circ k,
\end{array}$$
hence $P$ is a semigroup.
It is immediate that $2\Z$ and $\{2n+1\colon n\in \Z\}$ are the $\gL$-classes of $P$
and that $2\Z$ and $\{ 2n+1 \}$ $(n \in \Z)$ are its $\gH$- (and $\gR$-)classes,
which yields the desired conclusion.


\medskip

Of course, it may also happen that a semigroup fails to be $\gJ$-finite,
and thus belongs to neither of the other four classes:

\begin{ex} \label{E:monogenic}
 Both $(\N,+)$, the free monogenic semigroup, and $FIS_a$,
 the free monogenic inverse semigroup,
 have infinitely many $\gJ$-classes.
\end{ex}

\medskip

Next, we describe the behaviour of these properties with respect to the taking
of subsemigroups, homomorphic images, and direct products.

\begin{lma} \label{L:subsmgs}
 Let $\K\in \{\gH,\gL,\gR\}$, let $S\in \kfin$,
 and let $T\leq S$ be a regular subsemigroup of $S$.
 Then $T\in \kfin$.
\end{lma}

\begin{proof}
 The result is straightforward by Proposition~\ref{P:reg-subsmg},
 since $K_a^T=K_a^S \cap T$ yields that
 infinitely many $\K$-classes in $T$ implies infinitely many $\K$-classes in $S$.
\end{proof}

In this statement, not only the regularity of the subsemigroup
cannot be dropped, as also it cannot be extended to $\gD$ and $\gJ$,
as the following examples show:

\begin{ex}
 If $S$ is the infinite cyclic group generated by $a$
 and $T$ is its infinite monogenic subsemigroup
 (thus consisting only of the positive powers of $a$),
 then $\gL^S=S\times S$ while $\gL^T$ is the identity relation on $T$.
 Therefore, $S$ is $\gL$-finite whereas $T$ is not
 (and likewise for $\gR$ and $\gH$).
\end{ex}

\begin{ex} \label{E:CliPre-bisimple}
 By \cite[Theorem~8.55]{CliPre67}, any semigroup can be embedded in a
 (necessarily regular) bisimple semigroup (actually, monoid).
 Thus, $\gD$-finite semigroups can have non $\gD$-finite regular subsemigroups.
\end{ex}

Similarly,

\begin{ex} \label{E:CliPre-simple}
 Given that every semigroup can be embedded in a simple semigroup
 (cf.~\cite[Theorem~8.45]{CliPre67}),
 it follows that $\gJ$-finite semigroups can have  regular subsemigroups
 which are not $\gJ$-finite.
\end{ex}

Notice that, in the second of these last two examples,
the semigroup $S$ in which the original semigroup $T$ embeds into
is generated by $A\cup \{b,c\}$,
where $A$ is a generating set for $T$ and $b,c\notin T$,
and so $S$ is finitely generated whenever $T$ is.
Moreover, by \cite[Theorem~8.48]{CliPre67}
we have that $S$ is regular (respectively, inverse)
if and only if $T$ is regular (respectively, inverse).

\medskip

Another example is:

\begin{ex}
 Let $X=\Z\times \N_0$ and write, for all $r \in \Z$ and $s\geq 0$,
 $$V(r,s)=\{(x,y)\in X \colon s\leq y\leq x+s-r\} \, ,$$
 that is, the unbounded region depicted in Figure~\ref{F:Vsets}.
 \begin{figure}[!h]
  \[
   \parbox{4cm}{
    \begin{pspicture}(0,0)(4,4.5)
    \pspolygon[linestyle=none,fillstyle=solid,fillcolor=liggray](2,0)(2,2.5)(2.5,2.5)(2.5,2)(3,2)(3,1.5)(3.5,1.5)(3.5,1)(4,1)(4,0)(2,0)
    \psline[linewidth=.4pt](0,0)(0,4)
    \psline[linewidth=1.2pt](2,0)(2,2.5)(2.5,2.5)(2.5,2)(3,2)(3,1.5)(3.5,1.5)(3.5,1)(4,1)
    \uput[l](0,2.25){\small $r$}
    \uput[u](2.25,2.5){\small $s$}
    \psline[linewidth=.4pt,linestyle=dashed,dash=1pt 1pt](0,2.5)(2,2.5)
    \psline[linewidth=.4pt,linestyle=dashed,dash=1pt 1pt](0,2)(2,2)
   \end{pspicture}}
  \]
  \caption{The set $V(r,s)$}\label{F:Vsets}
 \end{figure}
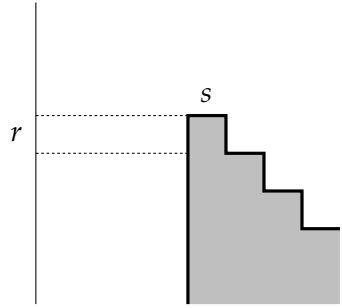
 Consider the bijections $\varphi\colon V(0,0)\to V(0,1)$,
 defined by $(x,y)\varphi=(x,y+1)$,
 and $\psi\colon X\to X$, defined by $(x,y)\psi=(x+1,y)$,
 and the inverse subsemigroups of the symmetric inverse semigroup $I_X$
 $$ T=\gen{\varphi} \; \text{ and } \; S=\gen{\varphi,\psi} \, .$$
 Then $S$ has two $\gJ$-classes,
 whereas its (regular) subsemigroup $T$ has infinitely many $\gJ$-classes.
\end{ex}

We begin by showing that the set
$\VV = \{ V(r,s) \colon r \in \mathbb{Z} \, , \; s \geq 0\}$
is closed under intersection.
So let $V(r,s)$ and $V(r',s')$ be sets in $\VV$,
for some $r,r' \in \mathbb{Z}$ and $s,s' \geq 0$,
and, without loss of generalization, assume that $s' \geq s$.
It is straightforward to check that
\begin{align*}
 V(r,s) \cap V(r',s')
    &= V(r,s) \cap (\mathbb{Z} \times [s',+\infty[) \cap V(r',s') \\
    &= V(r+s'-s,s') \cap V(r',s') \\
    &= V((r+s'-s)\vee r',s') \in \mathcal{V} \, .
\end{align*}

Next, we show that $T$ has infinitely many $\gJ$-classes;
indeed, we show that $T$ is isomorphic to the monogenic free inverse semigroup.
Let $\theta \colon FIS_a \to T$ be the epimorphism defined by $a\theta=\varphi$.
Since $FIS_a$ is combinatorial, checking that $\theta$ is injective
on $E=E(FIS_a)$ yields that $\theta$ is an isomorphism.
On the one hand, every $e\in E$ admits a unique representation in the form $a^{-r}a^{r+s}a^{-s}$,
with $r,s\geq 0$ not both $0$.
On the other hand, $\varphi^n\colon V(n-1,0)\to V(n-1,n)$, for all $n\geq 1$.
In fact,
\[
 \dom \varphi^2 = \big(\im \varphi \cap \dom \varphi\big)\varphi\inv
  = (V(0,1) \cap V(0,0))\varphi\inv
  = V(1,1)\varphi\inv
  = V(1,0)
\]
and
\[
 \im \varphi^2 = (\im \varphi \cap \dom \varphi)\varphi
  = V(1,1)\varphi
  = V(1,2)
\]
and the conclusion follows by induction.
Thus,
$$ \varphi^n\varphi^{-n} = \id|_{\dom \varphi^n} = \id|_{V(n-1,0)} \quad \text{ and } \quad
    \varphi^{-n}\varphi^n = \id|_{\im \varphi^n} = \id|_{V(n-1,n)} \, ,$$
for all $n\geq 1$, and so
\begin{align*}
 \varphi^{-r}\varphi^{r+s}\varphi^{-s}
 &= \id|_{V(r-1,r)} \; \id|_{V(s-1,0)} \\
 &= \id|_{V(r-1,r) \cap V(s-1,0)} \\
 &= \id|_{V(r+s-1,r)} \, ,
\end{align*}
for all $r,s\geq 1$ --- in fact, for all $r,s\geq 0$ not both $0$.
Therefore, the map $(r,s)\mapsto \dom (\varphi^{-r}\varphi^{r+s}\varphi^{-s})$
is injective, and so is $\theta|_E$.
Hence $T$ is a monogenic free inverse semigroup.
This ensures that $T$ has infinitely many $\gD$-classes,
and so infinitely many $\gJ$-classes as well.

In order to show that $S$ has two $\gJ$-classes,
we start by proving that
\begin{equation} \label{Eq:idps-S}
 E(S)=\{\id|_{V(r,s)} \colon r \in \Z,\; s\geq 0\} 
\cup \{\id|_{X} \} \, .
\end{equation}

Clearly, $\id|_X = \psi\psi\inv \in E(S)$ and,
if $r \in \mathbb{Z}$ and $s \geq 0$, we have
\begin{align*}
 \dom(\psi^{s-r-1}\varphi^{-s})
 &= (\im\psi^{s-r-1} \cap \dom\varphi^{-s})\psi^{-s+r+1} \\
 &= (V(s-1,s))\psi^{-s+r+1} \\
 &= V(r,s) \, ,
\end{align*}
hence $\psi^{s-r-1}\varphi^{-s}\varphi^{s}\psi^{-s+r+1} = \id|_{V(r,s)}$.
Thus, the right hand side of (\ref{Eq:idps-S}) is contained in $E(S)$.

Now we know by that $\VV$ is closed under intersection,
hence so is $\VV \cup \{ X \}$.
It follows that $\dom \alpha \in \VV \cup \{ X \}$ for every
$\alpha\in \{\varphi,\varphi\inv,\psi,\psi\inv\}$.
Since $\dom (\phi \sigma) = (\im \phi \cap \dom \sigma)\phi\inv$
whichever the injective mappings $\phi$ and $\sigma$
and since the preimage of any $V\in \VV \cup \{ X \}$ under either $\varphi$ or $\psi$
is still in $\VV \cup \{ X \}$,
we conclude that $\dom \alpha \in \VV \cup \{ X \}$ for every
$\alpha\in S=\gen{\varphi,\psi}$.
Therefore (\ref{Eq:idps-S}) holds.

Finally, for all $r\in \Z$ and $s\geq 0$, we can construct a bijection
$$V(r,s) \mapright{\psi^{s-r-1}} V(s-1,s) \mapright{\varphi^{-s}} V(s-1,0) \mapright{\psi^{1-s}} V(0,0),$$
hence all the idempotents in $\{\id|_{V(r,s)} \colon r \in \Z,\; s\geq 0\}$
belong to the same $\gJ$-class.
The only other idempotent, $\id|_X$, belongs clearly to a different $\gJ$-class,
so $S$ has two $\gJ$-classes as claimed.

\medskip

The next result is an easy consequence of the fact that (all) Green's relations are preserved
under homomorphism.
For that reason, this property holds, unlike the previous one, also for $\gD$ and $\gJ$.

\begin{lma} \label{L:hom-images}
 Let $\K\in \{\gH,\gL,\gR,\gD,\gJ\}$, let $S\in \kfin$, and let $\psi\colon S\to T$ be a morphism.
 Then $S\psi \in \kfin$.
\end{lma}

\begin{proof}
 Suppose there existed an infinite set $I$ such that
 $S\psi/\K^{S\psi}=\cup_{i\in I} K^{S\psi}_{a_j\psi}$,
 where $K^{S\psi}_{a_i\psi}\neq K^{S\psi}_{a_j\psi}$ whenever $i\neq j$.
 But then all the $\K^S$-classes $K^S_{a_i}$, with $i\in I$, would be distinct,
 for if $(a_i,a_j) \in \K^S$ for some $i,j\in I$ then also $(a_i\psi,a_j\psi) \in \K^{S\psi}$,
 a contradiction.
\end{proof}

\begin{lma} \label{L:finite-products}
 Let $\K\in \{\gH,\gL,\gR,\gD,\gJ\}$, let $I$ be a finite set,
 and let $\{S_i\colon i\in I\}\subset \kfin$,
 where each $S_i$ is either regular or a monoid.
 Then $\Pi_{i\in I} S_i \in \kfin$.
\end{lma}

\begin{proof}
 We prove the claim for a direct product $S\times T$,
 its generalization to the arbitrary case being straightforward.
 
 So let $S,T \in \kfin$.
 It is easy to check that, since both $S$ and $T$ are either regular or monoids,
 we have $(a,x)\K^{S\times T} (b,y)$ if and only if $a\K^S b$ and $x\K^T y$,
 for all $(a,x),(b,y) \in S\times T$.
 In order to obtain a contradiction,
 suppose $K_{(a_i,x_i)}^{S\times T}$, with $i\in I$,
 are infinitely many distinct $\K$-classes of $S\times T$.
 Consider the relations $\sim_S$ and $\sim_T$ defined, for all $i,j\in I$, by
 $$i\sim_S j \iff K^S_{a_i}=K^S_{a_j} \quad \text{ and } \quad
    i\sim_T j \iff K^T_{x_i}=K^T_{x_j} \, ,$$
 which are easily seen to be equivalences.
 Since both $S,T \in \kfin$,
 we have that $|I/\sim_S|<\infty$ and $|I/\sim_T|<\infty$.
 Thus, as $I$ is infinite, there exist distinct $i,j\in I$ such that
 $i\sim_S j$ and $i\sim_T j$,
 in which case $K^{S\times T}_{(a_i,x_i)}=K^{S\times T}_{(a_j,x_j)}$,
 a contradiction.
\end{proof}

As one might expect, the previous result no longer holds for infinite direct products.
Next, we provide one such example.

\begin{ex}
 For each $p\in \N$, let $\N_p=Mon\gen{a \mid a^p=a^{p+1}}$ be the monogenic monoid
 of index $p$ and period $1$.
 Since each monoid $\N_p$ is finite, we have that they are all $\gJ$-finite.
 The infinite direct product $\Pi_{p\in \N} \N_p$, however, is not.
\end{ex}

Let $S=\Pi_{p\in \N} \N_p$, fix $k \in \N$,
and take
$$x=(a,a^2,\ldots,a^k,a^k,a^k,\ldots) \, , \;
y=(a,a^2,\ldots,a^k,a^{k+1},a^{k+1},\ldots) \in S \, .$$
Since there are no $u,v \in \N_{k+1}$ such that $a^k=ua^{k+1}v$,
we have that $(a^k,a^{k+1})\notin \gJ^{\N_{k+1}}$.
Therefore, $(x,y)\notin \gJ^S$.
It follows that $S$ has infinitely many $\gJ$-classes.

\medskip

Also, the semigroups in the direct product do have to be either regular or monoids:
\begin{ex}
 Let $S=\{a_n \colon n \in \N\}$, where $a_i\neq a_j$ for any $i$ and $j$,
 be a right zero semigroup
 and $T=\{0,b\}$, where $b\neq 0$, be a null semigroup.
 Then both $S$ and $T$ are $\gR$-finite, whereas $S\times T$ is not.
\end{ex}

That $S$ and $T$ are $\gR$-finite is obvious.
However, $(a_m,b) \gR (a_n,b)$ if and only if $m=n$,
as, no matter if $z=0$ ou $z=b$,
we have $(a_m,b)=(a_n,b)(a_p,z)=(a_p,0)$ implying $b=0$,
which is a contradiction.
Therefore, $S\times T$ has infinitely many $\gR$-classes.

\medskip

Together with Examples~\ref{E:bicyclic} and \ref{E:monogenic},
Lemma~\ref{L:finite-products} provides yet another example that what holds for
$\gH$, $\gL$ and $\gR$ needs not hold for $\gD$ and $\gJ$,
namely the not passing of $\gD$-finiteness to regular subsemigroups
(cf.~Lemma~\ref{L:subsmgs}).
In fact, it is well-known that $FIS_a$ embeds in $B\times B$
and, while $B\times B$ is $\gD$-finite by
Example~\ref{E:bicyclic} and Lemma~\ref{L:finite-products},
its regular subsemigroup $FIS_a$ is not even $\gJ$-finite.

\section{Locally Green-finite varieties} \label{sec:vars}

Recall that a variety $\V$ is said to be \emph{locally finite} if all its finitely generated
members are finite.
These include, for example, the \emph{finitely generated} varieties,
that is, varieties generated by a finite family of finite algebras
(cf.~\cite{Mal70} or \cite{Bez01}).

\begin{defn}
 Given a variety $\V$ and $\K\in \{\gH,\gL,\gR,\gD,\gJ\}$,
 we say that $\V$ is \emph{locally $\K$-finite} if each finitely generated
 semigroup belonging to $\V$ is $\K$-finite.
\end{defn}

\begin{rmk} \label{R:finiteness-hierarchy}
 \begin{enumerate}
  \item Clearly, any locally $\gH$-finite variety is locally $\K$-finite, \linebreak
   whichever the Green's relation $\K$.
   Similarly, every locally $\gL$- or $\gR$-finite variety is locally $\gD$-finite,
   and every locally $\gD$-finite variety is locally $\gJ$-finite.
  \item It is also straightforward that every locally finite variety
   is locally $\gH$-finite (and, thus, locally $\gL$-, $\gR$-, $\gD$-,
   and $\gJ$-finite as well).
   Such is the case for the variety of bands (cf.~\cite[Theorem~4.5.3]{How95})
   and, as noted above, for every finitely generated variety.
 \end{enumerate}
\end{rmk}

As a matter of fact, for varieties of semigroups,
of completely regular semigroups, and of inverse semigroups,
being locally $\gD$-finite is equivalent to being
locally $\gJ$-finite:

\begin{theorem} \label{T:locDfin=locJfin}
 Let $\V$ be a variety of semigroups, a variety of completely regular semigroups
 or a variety of inverse semigroups.
 Then $\V$ is locally $\gJ$-finite if and only if it is locally $\gD$-finite.
\end{theorem}

\begin{proof}
 In view of the comments above, only the direct implication needs proven.
 So suppose $\V$ is locally $\gJ$-finite.
 
 In case $\V$ is a variety of completely regular semigroups,
 the result follows from Theorem~\ref{T:D=JinCR}.
 
 Now suppose $\V$ is a locally $\gJ$-finite variety of semigroups
 and let $S$ be a finitely generated semigroup in $\V$.
 We claim that $S$ is periodic.
 In order to obtain a contradiction, suppose $a\in S$ has infinite order.
 Then $\gen{a}$ is a subsemigroup of $S$, and so $\gen{a}\in \V$.
 In addition, $\gen{a}$ is finitely generated and has infinitely many $\gJ$-classes,
 as $\gen{a} \approx (\N,+)$.
 But this is a contradiction, since $\V$ is locally $\gJ$-finite by assumption.
 Thus $S$ is indeed periodic and so $\gD=\gJ$ in $S$ by Proposition~\ref{P:D=J-epigroup}.
 As $S$ is finitely generated, then $S$ is $\gJ$-finite by assumption.
 Therefore $S$ is $\gD$-finite as well and the conclusion follows.
 
 Finally, let $\V$ be a variety of inverse semigroups
 and let $S$ be a finitely generated semigroup from $\V$.
 Again, let $a\in S$ and consider the monogenic inverse semigroup $\gen{a}$.
 By \cite[Corollary~4]{Bropp}, there are only four possibilities:
 (i) $a$ has finite order;
 (ii) $\gen{a}$ has an infinite cyclic subgroup;
 (iii) $\gen{a}$ has a subsemigroup isomorphic to the bicyclic monoid;
 (iv) $\gen{a}$ is the monogenic free inverse semigroup.
 Cases (iii) and (iv) can be readily excluded,
 since they both imply that $\V$ would contain a
 non locally $\gJ$-finite variety,
 namely the variety generated by the bicyclic monoid,
 equivalently, by the monogenic free inverse semigroup (cf.~Theorem~\ref{T:FISa=B}).
 Therefore, one of (i) or (ii) must hold,
 and, in either case, $S$ is an epigroup (possibly, even a periodic semigroup).
 Again by Proposition~\ref{P:D=J-epigroup}, $\gD=\gJ$ in $S$ 
 and so, since by assumption $S$ is $\gJ$-finite,
 we may conclude that it is $\gD$-finite as well.
\end{proof}


\medskip

For a nontrivial variety $\V$ and a positive integer $n$,
denote by $F_n(\V)$ the $n$-generated free object on $\V$.
A simple consequence of Lemma~\ref{L:hom-images} is:

\begin{prop} \label{P:VveeW}
 Let $\K\in \{\gH,\gL,\gR,\gD,\gJ\}$.
 A variety $\V$ of semigroups or of $I$-semigroups is locally $\K$-finite
 if and only if each $F_n(\V)$ is $\K$-finite.
\end{prop}

\begin{proof}
 The direct implication is an immediate consequence of the definition.
 As for the converse, it follows straightforwardly
 from the universal property of free objects ---
 every finitely generated semigroup in $\V$ is a homomorphic image of
 some $F_n(\V)$ --- and from Lemma~\ref{L:hom-images}.
\end{proof}


\medskip

\begin{prop} \label{P:VveeW}
 Let $\V$ and $\W$ be varieties of $I$-semigroups.
 If $\V$ and $\W$ are locally $\K$-finite, with $\K\in \{\gH,\gL,\gR\}$,
 then so is $\V \vee \W$.
\end{prop}

\begin{proof}
 Fix $\K\in \{\gH,\gL,\gR\}$.
 
 Let $S\in \V \vee \W$ be a finitely generated semigroup,
 say $S=\gen{x_1,\ldots,x_n}$.
 Then there exist $V\in \V$, $W\in \W$, $T\leq V\times W$ and an epimorphism
 $\varphi\colon T\twoheadrightarrow S$.
 Since $\varphi$ is onto,
 we can take $a_i\in T$ such that $a_i\varphi=x_i$ for every $i$;
 let $R$ be the $I$-subsemigroup of $T$ generated by $\{a_1,\ldots,a_n\}$.
 Also, for each $i$, we have $a_i=(v_i,w_i)$ for some $v_i\in V$ and $w_i\in W$;
 let $V'=\gen{v_1,\ldots,v_n}$ and $W'=\gen{w_1,\ldots,w_n}$.
 Thus $V'\in \V$ and $W'\in \W$ are $\K$-finite,
 and so is $V'\times W'$ by Lemma~\ref{L:finite-products}.
 Thus, since it is a regular subsemigroup of $V'\times W'$,
 $R$ is $\K$-finite by Lemma~\ref{L:subsmgs}.
 Therefore $S=R\varphi$ is $\K$-finite by Lemma~\ref{L:hom-images}.
\end{proof}

As a consequence,

\begin{cor} \label{C:varS-K-fin}
 If $S$ is a $\K$-finite  $I$-semigroup, for $\K\in \{\gH,\gL,\gR\}$,
 then the variety of $I$-semigroups $\gen{S}$ is locally $\K$-finite.
\end{cor}

\begin{proof}
 Let $T\in \gen{S}$ be finitely generated.
 Then, by Theorem~\ref{T:var-characterization},
 there exists an epimorphism $\psi\colon R \to T$,
 where $R\leq \Pi_{\lambda \in \Lambda} S_{\lambda}$
 with all $S_{\lambda}=S$.
 Since $T$ is finitely generated,
 the set $\Lambda$ can be taken finite.
 But then $\Pi_{\lambda \in \Lambda} S_{\lambda}$ is $\K$-finite
 by Lemma~\ref{L:finite-products},
 the subsemigroup $R$, being regular,
 is $\K$-finite by Lemma~\ref{L:subsmgs},
 and $T$ is $\K$-finite by Lemma~\ref{L:hom-images}.
\end{proof}

This conclusion fails both for $\gD$ and $\gJ$.
For the first one, we have for instance the bicyclic monoid,
which is $\gD$-finite (cf. Example~\ref{E:bicyclic})
whereas the variety it generates is not locally $\gD$-finite,
as it contains the free monogenic inverse semigroup
(cf. Example~\ref{E:monogenic} and Theorem~\ref{T:FISa=B});
for the second, take the inverse monoid $M$ from Example~\ref{E:a^nb}.
Then $M$ is $\gJ$-finite, but $\gen{M}$ is not locally $\gJ$-finite
as it contains the inverse subsemigroup $\gen{a}$,
which is, again, the free monogenic inverse semigroup.

\subsection{Varieties of completely regular semigroups} \label{subsec:CR}

The lattice of varieties of completely regular semigroups is arguably the most
thoroughly investigated.
In addition, the classes of semigroups involved are quite well behaved
with respect to Green's relations.
As a consequence, a complete characterization of some of the most important
subvarieties of completely regular semigroups can be obtained.

\medskip

We begin with a simple observation, followed by a useful consequence.

\begin{lma} \label{L:Yfinite}
 If $S=(\B;S_{\beta})$ is a finitely generated cryptogroup,
 then $\B$ is a finite band.
\end{lma}

\begin{proof}
 Taking the canonical epimorphism $S\twoheadrightarrow S/\gH$,
 it is immediate that $S$ being finitely generated implies that so is $S/\gH$.
 Thus, $S/\gH$ is a finitely generated band.
 Therefore, $S/\gH$ is finite, since the variety of bands is locally finite.
\end{proof}

\begin{cor} \label{C:BG-Hfinite}
 The variety $\mathbf{BG}$ of all cryptogroups is locally $\gH$-finite.
\end{cor}

\begin{proof}
 Let $S=(\B;S_{\beta})$ be a finitely generated cryptogroup.
 Then, in view of the previous lemma, $\B$ is a finite band.
 Since $|S/\gH|=|\B|$, we conclude that $S$ has finitely many $\gH$-classes.
 Hence $S$ is $\gH$-finite.
\end{proof}

A similar argument shows that the variety $\CR$ of all completely regular semigroups,
and thus all its subvarieties, are, at least, locally $\gD$-finite;
our goal is thus to refine this characterization.

First, we return to Example~\ref{E:Pedro1}.

\begin{ex} \label{E:Pedro2}
 (Continued from Example~\ref{E:Pedro1}.)
 The $\gL$-finite, but not $\gR$-finite, semigroup $P$ from
 Example~\ref{E:Pedro1} is a finitely generated completely regular semigroup;
 in fact, it is a finitely generated member of $\mathbf{ROL^*}$,
 that is, a finitely generated regular orthogroup in which
 $\gL$ is a congruence.
\end{ex}

The fact that $2$ generates the subgroup $2\Z$ together with the fact that
$(2n) \circ 1 = 2n+1$ implies that $P$ is generated as a completely regular semigroup
by $\{ 1,2\}$.

Recall that its $\gR$- and $\gH$-classes coincide: they are $2\Z$ and $\{ 2n+1 \}$,
for each $n \in \Z$.
Therefore, $P$ is a completely regular semigroup, since each $\gH$-class is a subgroup.

To see that $P$ belongs to $\mathbf{ROL^*}$,
we check that is satisfies the identity $x(y^0z)^0x=xy^0x^0z^0x$.
So let $m,n,p \in P$.
If $m$ is odd or if $m$ is even and $n$ is odd
(notice that, in the second case, $H_n=\{n\}$ and thus $n^0=n$),
we have $m\circ (n^0\circ p)^0\circ m = m\circ n^0\circ m^0\circ p^0\circ m$.
So suppose both $m$ and $n$ are even.
Then $m^0=n^0=0$ and:
if $p$ is odd,
$$m\circ (n^0\circ p)^0\circ m = m\circ (0\circ p)^0\circ m
  = m\circ p^0\circ m
  = m\circ p\circ m
  = m+p $$
and
$$m\circ n^0\circ m^0\circ p^0\circ m = m\circ 0\circ 0\circ p\circ m
  = m+p \, ;$$
if $p$ is even,
$$m\circ (n^0\circ p)^0\circ m = m\circ (0\circ p)^0\circ m
  = m\circ p^0\circ m
  = m\circ 0\circ m
  = 2m $$
and
$$m\circ n^0\circ m^0\circ p^0\circ m = m\circ 0\circ 0\circ 0\circ m
  = 2m \, .$$
Therefore,
$m\circ (n^0\circ p)^0\circ m = m\circ n^0\circ m^0\circ p^0\circ m$
for every $m,n,p \in P$, and so $P \in \mathbf{ROL^*}$.
(Clearly $P$ does not belong to $\mathbf{ROR^*}$,
as $\gR$ is not a congruence on $P$:
for instance, $(0,2)\in \gR$ but $(0\circ 1,2\circ 1)=(1,3)\notin \gR$.)


\medskip

The next result summarizes the behaviour of some of the classical subvarieties
of $\CR$ as to being locally $\gH$-, $\gL$-, and $\gR$-finite,
or simply locally $\gD$-finite.
Its conclusions are also pictured in Figure~\ref{F:CRclassification}.

\begin{theorem} \label{T:CRclassification}
 \begin{itemize}  
  \item[(i)] All the subvarieties of $\mathbf{BG}$ are locally $\gH$-finite.
  \item[(ii)] Every variety in the interval $[{\bf LRO},{\bf ROL}^*]$
   is locally $\gL$-finite but not locally $\gR$-finite.
  \item[(iii)] Every variety in the interval $[{\bf RRO},{\bf ROR}^*]$
   is locally $\gR$-finite but not locally $\gL$-finite.
  \item[(iv)] Every variety in the interval $[{\bf RO},{\bf CR}]$
  is locally $\gD$-finite but neither $\gL$-finite nor $\gR$-finite.
 \end{itemize}
\end{theorem}

\begin{proof}
 (i) By Corollary~\ref{C:BG-Hfinite}, we know that $\mathbf{BG}$ is locally $\gH$-finite;
 that the same is true for all if its subvarieties follows trivially.
 
 (ii) Let $S\in \mathbf{ROL^*}$ be a finitely generated semigroup.
 As a consequence of Theorem~\ref{T:ROLstar},
 we have that $S$ embeds in a direct product $L\times B$
 of a left regular orthogroup $L$ and a right regular band $B$.
 Since $(x,a)\gL^S (y,b)$ if and only if $x\gL^L y$ and $a\gL^B b$,
 we have that $|S/\gL^S|\leq |L/\gL^L|\times |B/\gL^B|$.
 More precisely, $|S/\gL^S|\leq |L/\gL^L|\times |S\pi_B/\gL^{S\pi_B}|$,
 where $S\pi_B$ is the subband of $B$ generated by the projection of $S$ into $B$.
 Now, since $\gL=\gD$ on $L$ (cf.~comment regarding Lemma~\ref{L:ROL}),
 then $L$ has finitely many $\gL$-classes.
 Also $S\pi_B$ has finitely many $\gL$-classes, since it is a finitely generated band,
 and so itself finite.
 Therefore $S$ has finitely many $\gL$-classes.
 Hence, $\mathbf{ROL^*}$ is locally $\gL$-finite ---
 but not locally $\gR$- (and $\gH$-)finite, by virtue of Example~\ref{E:Pedro2}.
 
 As for $\mathbf{LRO}$, it is at least locally $\gL$-finite
 as it is contained in $\mathbf{ROL^*}$.
 Were it also locally $\gR$-finite, then it would be locally $\gH$-finite,
 along with $\mathbf{ROL^*}=\mathbf{LRO} \vee \mathbf{ROBG}$ as a consequence of
 Proposition~\ref{P:VveeW}, a contradiction.
 
 (iii) Dual from (ii).
 
 (iv) Since $\mathbf{RO}$, the variety of regular orthogroups,
 contains $\mathbf{ROL^*}$ as a subvariety,
 by (ii) $\mathbf{RO}$ cannot be locally $\gR$-finite
 --- and, dually by (iii), neither can it be locally $\gL$-finite.
 Thus, $\mathbf{RO}$ is just locally $\gD$-finite,
 and since so is $\CR$, the same happens for all the varieties in the interval
 $[\mathbf{RO},\CR]$.
\end{proof}


Notice that, by (i), the particularly important varieties
of \emph{completely simple} semigroups, that is, completely regular semigroups
in which $\gJ=S\times S$,
and of Clifford semigroups are all locally $\gH$-finite.

\begin{figure}
\[
 \parbox{15cm}{
 \begin{pspicture}(-2,-1)(13,9)
  \psline[linewidth=.4pt](2.5,.8)(2.5,-.2)
  \psdot(1,2.2)
  \uput[l](1,2.2){\small $\mathbf{LRO}$}
  \psdot(1,3.2)
  \uput[l](1.1,3.2){\small $\mathbf{ROL^{\ast}}$}
  \psdot(2.5,.8)
  \uput[l](2.5,.8){\small $\mathbf{ROBG}$}
  \psdot(2.5,-.2)
  \uput[l](2.5,-.2){\small $\mathbf{SG}$}
  \psdot(2.5,4.2)
  \uput[u](2.5,4.2){\small $\mathbf{RO}$}
  \psdot(4,2.2)
  \uput[r](4,2.2){\small $\mathbf{RRO}$}
  \psdot(4,3.2)
  \uput[r](4,3.2){\small $\mathbf{ROR^{\ast}}$}
  \psdot(7.5,2.8)
  \uput[r](7.5,2.8){\small $\mathbf{OBG}$}
  \psdot(5,5.2)
  \uput[u](5,5.2){\small $\mathbf{OG}$}
  \psdot(7.5,6.2)
  \psdot(7.5,7.2)
  \uput[u](7.5,7.2){\small $\mathbf{CR}$}
  \psdot(10,5.2)
  \uput[r](10,5.2){\small $\mathbf{BG}$}
  \psline[linewidth=.4pt](1,2.2)(1,3.2)(2.5,4.2)(4,3.2)(4,2.2)
  \psline[linewidth=.4pt](1,3.2)(2.5,.8)(4,3.2)
  \psline[linewidth=.4pt](2.5,4.2)(5,5.2)(7.5,6.2)(7.5,7.2)
  \psline[linewidth=.4pt](7.5,6.2)(10,5.2)(7.5,2.8)(5,5.2)
  \psline[linewidth=.4pt](2.5,.8)(7.5,2.8)
  \psline[linewidth=.4pt](1,2.2)(2.5,-.2)(4,2.2)
  \uput[u](8.8,7.8){locally $\gD$-finite}
  \pscurve[linewidth=.8pt,linestyle=dashed,dash=2pt 2pt](0,4.7)(2.5,3.8)(7.5,5.8)(10.4,8)
  \uput[d](11,4.5){locally $\gH$-finite}
  \pscurve[linewidth=.8pt,linestyle=dashed,dash=2pt 2pt](11.4,5.4)(10,5.6)(5.5,2.3)(4,1.6)(0,1)
  \uput[l](.8,1.6){locally $\gL$-finite}
  \psccurve[linewidth=.8pt,linestyle=dashed,dash=2pt 2pt,](1,1.8)(1.5,2.7)(1,3.6)(.5,2.7)
  \uput[l](3.6,3){locally}
  \uput[l](3.6,2.6){$\gR$-finite}
  \psccurve[linewidth=.8pt,linestyle=dashed,dash=2pt 2pt](4,1.8)(4.5,2.7)(4,3.6)(3.5,2.7)
\end{pspicture}}
\]
\caption{Classification of some completely regular varieties}\label{F:CRclassification}
\end{figure}
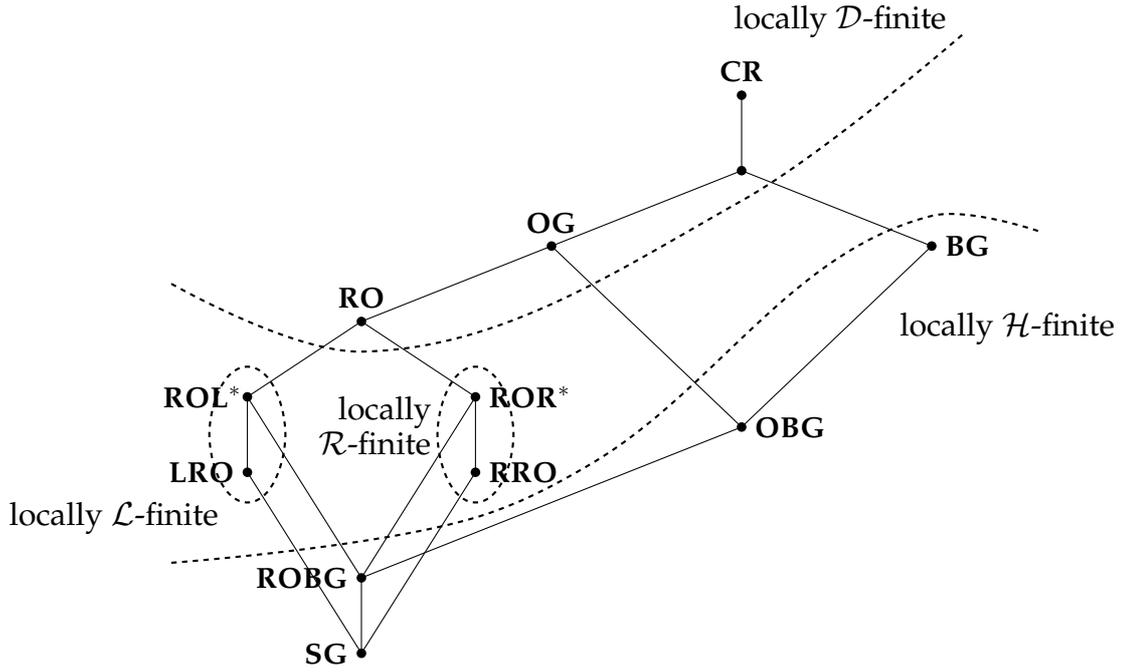

\medskip

\subsection{Varieties of inverse semigroups} \label{subsec:INV}

Since in any inverse semigroup $S$ and in any $\gD$-class $D$ of $S$
there is a bijection between the $\gL$-classes and the $\gR$-classes in $D$,
it follows that an inverse semigroup is $\gL$-finite if and only if it is $\gR$-finite.
Thus,

\begin{prop}
 For any variety $\V$ of inverse semigroups, the following are equivalent:
 \begin{itemize}
  \item[(i)] $\V$ is locally $\gH$-finite.
  \item[(ii)] $\V$ is locally $\gL$-finite.
  \item[(iii)] $\V$ is locally $\gR$-finite.
 \end{itemize}
\end{prop}

\begin{proof}
 We have (i)~$\Rightarrow$~(ii) as $|S/\gL|\leq |S/\gH|$ in any semigroup,
 (ii)~$\Leftrightarrow$~(iii) by the preceding observation,
 and consequently (iii)~$\Rightarrow$~(i) by Remark~\ref{R:class-poset}.
\end{proof}

\medskip

Recall from Theorem~\ref{T:locDfin=locJfin} that a variety of inverse semigroups
is locally $\gD$-finite if and only if it is locally $\gJ$-finite.
Notice that it follows from the proof of this result that

\begin{cor}
 If $\V$ is a locally $\gJ$-finite variety of inverse semigroups,
 then $\V$ is a variety of epigroups.
\end{cor}

The converse, however, is not true:

\begin{ex} \label{E:C2}
 $\mathbf{C_2}$, the subvariety of $\I$ defined by the identity $x^2=x^3$,
 is a variety of epigroups, but is not locally $\gJ$-finite.
\end{ex}

Indeed, every $S \in \mathbf{C_2}$ is an epigroup since, for every $a \in S$,
we have $a^2 \in E(S)$ and therefore $a^2a^{-2} = a^{-2}a^2$.
Write $A = \{ a,b,c\}$ and let $T$ be defined by the inverse semigroup presentation
$${\rm InvS}\langle A \mid
    u^2 = 0\; (u \in (A \cup A^{-1})^+,\; \overline{u} \neq 1)\rangle.$$
Clearly, $T \in \mathbf{C_2}$. 

Let $\alpha$ be the (right) infinite square-free word due to
Morse and Hedlund in connection with their solution of the Burnside problem
with $n=2$ for semigroups (\cite{MoHe38}),
and let $\alpha_n$ denote the length $n$ prefix of $\alpha$ for every $n \geq 1$.
Clearly, $L(\alpha_n) = MT(\alpha_n) = S\Gamma(\alpha_n)$ for every $n$.
Suppose that $m < n$ and $\alpha_m \gJ \alpha_n$ in $T$.
Then there exist $v,w \in (A \cup A^{-1})^*$ such that
$\alpha_m = v\alpha_nw$ holds in $T$,
hence $S\Gamma(\alpha_n) \approx S\Gamma(v\alpha_mw)$.
It follows that $\alpha_m$ labels some path in $S\Gamma(\alpha_n) = L(\alpha_n)$,
a contradiction since $m > n$ and $\alpha_m,\alpha_n \in A^+$. 

Thus each $\alpha_n$, for $n \geq 1$, belongs to a distinct $\gJ$-classe of $T$
and so $T$ is not $\gJ$-finite.
Therefore $\mathbf{C_2}$ is not locally $\gJ$-finite.

\medskip

Next, we summarize the classification of some of subvarieties of $\I$
with respect to the properties under discussion.
Its conclusions are also displayed in Figure~\ref{F:Iclassification}.

\begin{theorem} \label{T:INVclassification}
 \begin{itemize}
  \item[(i)] The varieties $\gen{B_2^1}$ and $\gen{M_n}$, for any positive integer $n$,
  are locally $\gH$-finite.
   As a consequence, so are the subvarieties $\G\vee \gen{M_n}$,
   for any positive integer $n$, and all the subvarieties of $\G\vee \gen{B_2^1}$.
  \item[(ii)] The varieties $\mathbf{C_2}$ and $\gen{B}$,
   as well as all the varieties of inverse semigroups that contain either
   $\mathbf{C_2}$ or $\gen{B}$, are not locally $\K$-finite,
   for any $\K \in \{\gH,\gL,\gR,\gD,\gJ\}$.
 \end{itemize}
\end{theorem}

\begin{proof}
 (i) Since $B_2^1$, the $5$-element Brandt semigroup with an identity adjoined,
 and $M_n$, for $n$ positive integer, are finite semigroups,
 the varieties they generate are finitely generated,
 and thus locally $\gH$-finite, as mentioned in Remark~\ref{R:finiteness-hierarchy}~(2).
 
 That $\G\vee \gen{M_n}$ and  $\G\vee \gen{B_2^1}$ are
 locally $\gH$-finite follows from Proposition~\ref{P:VveeW},
 the last assertion being now immediate.
 
 
 (ii) By part (1) of Remark~\ref{R:finiteness-hierarchy},
 it suffices to show that these varieties fail to be locally $\gJ$-finite.
 
 That $\mathbf{C_2}$ is not locally $\gJ$-finite has been established
 in Example~\ref{E:C2}.
 
 As seen is Example~\ref{E:monogenic},
 the monogenic inverse semigroup is not $\gJ$-finite.
 Since $FIS_a$ is finitely generated and $\gen{B}=\gen{FIS_a}$,
 we conclude that $\gen{B}$ cannot be locally $\gJ$-finite.
 
 Again, the last claim is therefore immediate.
\end{proof}

\begin{figure}
\[
 \parbox{8cm}{
  \begin{pspicture}(0,0)(8,8)
   \psline[linewidth=.4pt](1.4,.2)(6.4,5.2)
   \psline[linewidth=.4pt](.2,1.4)(5.2,6.4)
   \psline[linewidth=.4pt](2.4,4.2)(5.2,7)
   \psline[linewidth=.4pt](3.6,2.4)(2.4,3.6)(2.4,4.2)
   \psline[linewidth=.4pt](4.6,3.4)(3.4,4.6)(3.4,5.2)
   \psline[linewidth=.4pt](5.6,4.4)(4.4,5.6)(4.4,6.2)
   \psline[linecolor=medgray,linewidth=.4pt](3,1.8)(4.6,2.8)(6.4,4.6)
   \psline[linecolor=medgray,linewidth=.4pt](1.8,3)(3.4,4)(5.2,5.8)
   \psline[linecolor=medgray,linewidth=.4pt,linestyle=dashed,dash=1pt 1pt](6.4,4.6)(7.2,5.4)
   \psline[linecolor=medgray,linewidth=.4pt](4.6,3.4)(4.6,2.8)(3.4,4)(3.4,4.6)
   \psline[linecolor=medgray,linewidth=.4pt](5.6,4.4)(5.6,3.8)(4.4,5)(4.4,5.6)
   \psline[linewidth=.4pt,linestyle=dashed,dash=1pt 1pt](6.4,5.2)(7.9,6.7)
   \psline[linewidth=.4pt,linestyle=dashed,dash=1pt 1pt](5.2,6.4)(6.7,7.9)
   \psline[linewidth=.4pt,linestyle=dashed,dash=1pt 1pt](5.2,7)(6.7,8.5)
   \psline[linecolor=medgray,linewidth=.4pt,linestyle=dashed,dash=1pt 1pt](5.2,5.8)(6.2,6.8)
   \psline[linewidth=.4pt,linestyle=dashed,dash=1pt 1pt](7.2,5.4)(7.9,6.1)
   \psline[linewidth=.4pt,linestyle=dashed,dash=1pt 1pt](6.2,6.8)(6.7,7.3)
   \psline[linecolor=red,linewidth=.5pt](3,1.8)(3.6,2.4)
   \psline[linecolor=red,linewidth=.5pt](1.8,3)(2.4,3.6)
   \psline[linecolor=red,linewidth=.5pt](4.6,2.8)(4.6,3.4)
   \psline[linecolor=red,linewidth=.5pt](5.6,3.8)(5.6,4.4)
   \psline[linecolor=red,linewidth=.5pt](3.4,4)(3.4,4.6)
   \psline[linecolor=red,linewidth=.5pt](4.4,5)(4.4,5.6)
   \psdot[linecolor=medgray](1.4,.2)
   \uput[dr](1.4,.2){\medgray \footnotesize $\T$}
   \psdot[linecolor=medgray](.2,1.4)
   \uput[l](.2,1.4){\medgray \footnotesize $\G$}
   \psline[linecolor=medgray,linewidth=.4pt](1.4,.2)(.2,1.4)
   \psdot[linecolor=medgray](2,.8)
   \uput[dr](1.9,.8){\medgray \footnotesize $\mathbf{SL}=\mathbf{C_1}$}
   \psdot[linecolor=medgray](.8,2)
   \uput[l](.8,2){\medgray \footnotesize $\G \vee \mathbf{SL}=\mathbf{SG}$}
   \psline[linecolor=medgray,linewidth=.4pt](.8,2)(2,.8)
   \psdot[linecolor=medgray](2.6,1.4)
   \uput[dr](2.5,1.4){\medgray \footnotesize $\mathbf{CSInv}=\gen{B_2}=\gen{M_2}$}
   \psdot[linecolor=medgray](1.4,2.6)
   \uput[l](1.4,2.6){\medgray \footnotesize $\G \vee \mathbf{CSInv}=\mathbf{SInv}$}
   \psline[linecolor=medgray,linewidth=.4pt](1.4,2.6)(2.6,1.4)
   \psdot[linecolor=medgray](3,1.8)
   \uput[r](2.9,1.7){\medgray \footnotesize $\left\langle\right. \!\!{B_2}^1 \!\! \left.\right\rangle$}
   \psdot[linecolor=medgray](1.8,3)
   \uput[l](1.85,3.1){\medgray \footnotesize $\G \vee \left\langle\right. \!\!{B_2}^1 \!\! \left.\right\rangle$}
   \psline[linecolor=medgray,linewidth=.4pt](3,1.8)(1.8,3)
   \psdot(3.6,2.4)
   \uput[l](3.65,2.45){\footnotesize $\mathbf{C_2}$}
   \psdot(2.4,3.6)
   \uput[l](2.4,3.6){\footnotesize $\G\vee \mathbf{C_2}$}
   \psdot(2.4,4.2)
   \uput[l](2.45,4.3){\footnotesize $[x^2\in G]$}
   \psdot[linecolor=medgray](4.6,2.8)
   \uput[r](4.6,2.8){\medgray \footnotesize $\gen{M_3}$}
   \psdot(4.6,3.4)
   \uput[l](4.65,3.4){\footnotesize $\mathbf{C_3}$}
   \psdot(3.4,4.6)
   \uput[r](3.4,4.6){\footnotesize $\G\vee \mathbf{C_3}$}
   \psdot[linecolor=medgray](3.4,4)
   \uput[r](3.4,4){\medgray \footnotesize $\G\vee \gen{M_3}$}
   \psdot(3.4,5.2)
   \uput[l](3.45,5.3){\footnotesize $[x^3\in G]$}
   \psdot[linecolor=medgray](5.6,3.8)
   \uput[r](5.6,3.8){\medgray \footnotesize $\gen{M_4}$}
   \psdot(5.6,4.4)
   \uput[l](5.65,4.45){\footnotesize $\mathbf{C_4}$}
   \psdot(4.4,5.6)
   \uput[r](4.4,5.6){\footnotesize $\G\vee \mathbf{C_4}$}
   \psdot[linecolor=medgray](4.4,5)
   \uput[r](4.4,5){\medgray \footnotesize $\G\vee \gen{M_4}$}
   \psdot(4.4,6.2)
   \uput[l](4.45,6.25){\footnotesize $[x^4\in G]$}
   \psdot(7.2,5.4)
   \uput[r](7.2,5.4){\footnotesize $\gen{B}=\gen{FIS_a}$}
   \uput[r](5.5,2){\medgray \footnotesize locally $\gH$-finite}
   \uput[l](1,6){\footnotesize not locally $\K$-finite}
   \psline[linecolor=medgray,linewidth=.4pt](1.4,.2)(3,1.8)
   \psline[linecolor=medgray,linewidth=.4pt](.2,1.4)(1.8,3)
 \end{pspicture}}
\]
\caption{Classification of some inverse varieties}\label{F:Iclassification}
\end{figure}
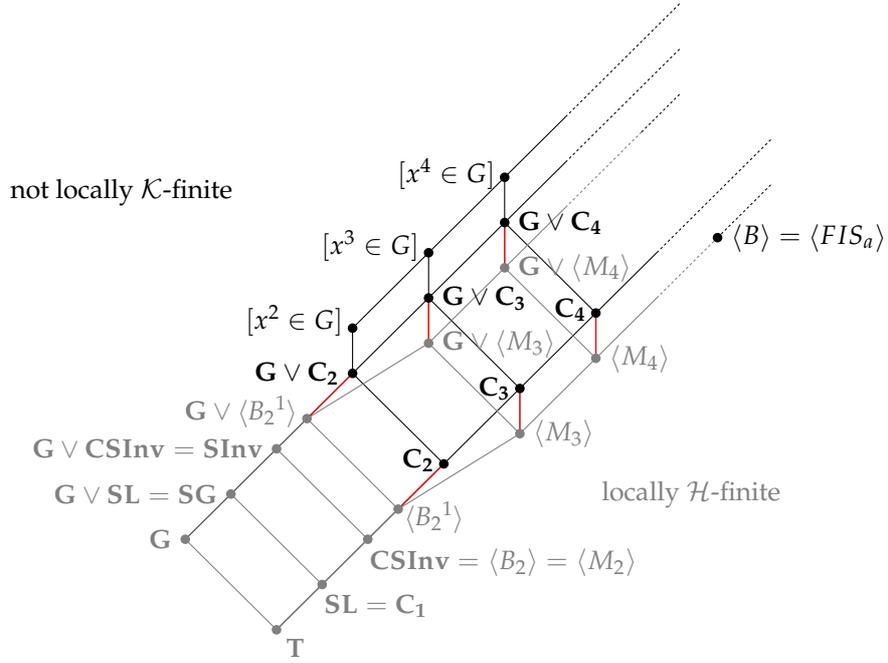

What happens in the open invervals $]\gen{B_2^1},\mathbf{C_2}[$
and $]\gen{M_n},\mathbf{C_n}[$ (for $n\ge 3$)
and whether or not there are strictly locally $\gD$-finite varieties
of inverse semigroups remains unanswered.
The fact that the variety $\gen{B_2^1}$ has no cover in the sublattice
of varieties of combinatorial inverse semigroups (see~\cite{Kad92})
indicates that for instance the first of these questions
may not be an entirely trivial one.

\subsection{Varieties of strict $I$-semigroups} \label{subsec:SIS}

Theorems~\ref{T:CRclassification} and \ref{T:INVclassification}
together answer the question for the case of the variety $\SI$.

\begin{cor}
 The variety $\SI$, and all its subvarieties, are locally $\gH$-finite.
\end{cor}

\begin{proof}
 By Corollary~\ref{C:SI-join}, $\SI=\mathbf{ONBG}\vee \gen{B_2}$.
 Since they are both locally $\gH$-finite,
 the first by Theorem~\ref{T:CRclassification}
 as $\mathbf{ONBG} \subseteq \mathbf{BG}$
 (in fact, $\mathbf{ONBG} \subseteq \mathbf{ROBG}$)
 and the second by Theorem~\ref{T:INVclassification},
 we conclude by Proposition~\ref{P:VveeW} that so is $\SI$.
\end{proof}

\subsection{Varieties of semigroups} \label{subsec:SMG}

The more complex nature of the lattice of all semigroup varieties
makes it harder to achieve a characterization as sucessful as that
of varieties of completely regular semigroups,
or even as complete as that of varieties of inverse semigroups.

We do know from Theorem~\ref{T:locDfin=locJfin} that every locally
$\gJ$-finite variety of semigroups is necessarily locally $\gD$-finite.
In this section, we list some other conclusions.

\medskip

One way of partitioning the lattice of all semigroup varieties
is the one described in \cite{SVV09}.
For the convenience of the reader, Figure~\ref{F:SEM} displays their exact,
very useful, diagram.

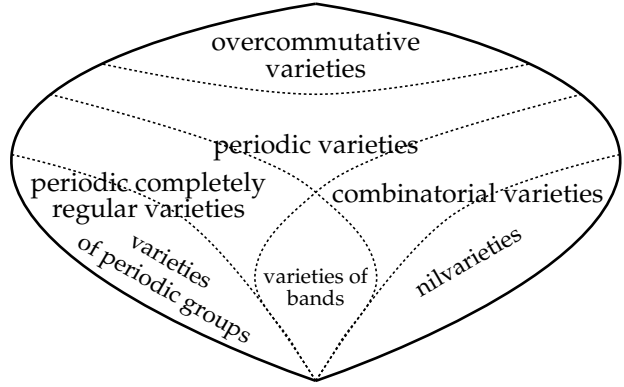
\begin{figure}[!h]
  \[
   \parbox{8cm}{
    \begin{pspicture}(-4,0)(4,5)
    \pscurve[linewidth=1pt](0,0)(4,3)(0,5)
    \pscurve[linewidth=1pt](0,0)(-4,3)(0,5)
    \pscurve[linewidth=.6pt,linestyle=dashed,dash=1pt 1pt](-2.8,4.2)(0,3.8)(2.8,4.2)
    \pscurve[linewidth=.6pt,linestyle=dashed,dash=1pt 1pt](0,0)(1.8,2.2)(4,3)
    \pscurve[linewidth=.6pt,linestyle=dashed,dash=1pt 1pt](0,0)(-1.8,2.2)(-4,3)
    \pscurve[linewidth=.6pt,linestyle=dashed,dash=1pt 1pt](0,0)(-.3,.48)(-.8,1.4)(0,2.5)(.9,3)(3.5,3.8)
    \pscurve[linewidth=.6pt,linestyle=dashed,dash=1pt 1pt](0,0)(.3,.48)(.8,1.4)(0,2.5)(-.9,3)(-3.5,3.8)
    \rput(0,4.5){\footnotesize overcommutative}
    \rput(0,4.15){\footnotesize varieties}
    \rput(0,3.1){\footnotesize periodic varieties}
    \rput(2,2.5){\footnotesize combinatorial varieties}
    \rput(-2.2,2.6){\footnotesize periodic completely}
    \rput(-2.2,2.25){\footnotesize regular varieties}
    \rput{30}(2,1.6){\scriptsize nilvarieties}
    \rput{-30}(-1.9,1.6){\scriptsize varieties}
    \rput{-30}(-2,1.2){\scriptsize of periodic groups}
    \rput(0,1.4){\tiny varieties of}
    \rput(0,1.1){\tiny bands}
   \end{pspicture}}
  \]
  \caption{The lattice of all semigroup varieties}\label{F:SEM}
 \end{figure}

Accordingly, this lattice can be split into the disjont union
of the ideal $\lat(\mathbf{Per})$ of all periodic varieties
and the coideal consisting of the varieties that include the variety
of all commutative semigroups, called \emph{overcommutative varieties}.

The second lot is easy to settle.
In fact, if a variety of semigroups includes every commutative semigroup,
then it includes $(\N,+)$.
Since this is a finitely generated semigroup which is not $\gJ$-finite,
as noted in Example~\ref{E:monogenic},
we conclude that such a variety is not locally $\gJ$-finite. 

The first, on the contrary, is not at all trivial.
Now, a variety is periodic if and only if it satisfies an identity of the form
$x^m=x^{m+n}$, or equivalently, if it is contained in some
\emph{Burnside variety} $\mathbf{B_{m,n}}=[x^m=x^{m+n}]$,
so it is only natural to begin with these.
Within Burnside varieties, we find the variety $\mathbf{C_2}=[x^2=x^3]$, which
as shown in Example~\ref{E:C2} is not locally $\gJ$-finite.
Since $\mathbf{C_2}=\mathbf{B_{2,1}} \subseteq \mathbf{B_{m,n}}$ whenever $m\ge 2$,
we conclude that all varieties $\mathbf{B_{m,n}}$ with $m\geq 2$ are not locally $\gJ$-finite.
In particular, the ideal $\lat(\mathbf{Comb})$ of $\lat(\mathbf{Per})$
consisting of all combinatorial varieties contains
non locally $\gJ$-finite varieties.

Another ideal of $\lat(\mathbf{Per})$, however,
contains only locally $\gD$-finite varieties:
the ideal $\lat(\mathbf{PCR})$ of all periodic completely regular varieties
(recall that a variety of semigroups consists of completely regular
semigroups if and only of it satisfies an identity of the form $x=x^{n+1}$
for some positive integer $n$).
Since $\lat(\mathbf{PCR})$ can be considered as a sublattice of $\lat(\mathbf{CR})$,
the lattice of all (unary) completely regular semigroups analysed in
Section~\ref{subsec:CR}, the conclusions obtained there can be used here.
In particular, the lattice of all varieties of bands consists only of
locally $\gH$-finite varieties
by Remark~\ref{R:finiteness-hierarchy}~(2),
as this is a locally finite variety.
For the same reason, a few other Burnside varieties,
namely $[x=x^3]$, $[x=x^4]$, $[x=x^5]$ and $[x=x^7]$, are locally $\gH$-finite,
since a variety $\mathbf{B_{1,n}}$ is locally finite
if and only if the corresponding group variety $\mathbf{B_n}=[x^n=1]$
is locally finite \cite{GR52} and this is true, at least,
for $n=2$, $3$, $4$ and $6$
(and false for large odd $n$) \cite{Neu67}.
As for $[x=x^6]$ and in $[x=x^n]$ with $n\ge 8$,
we do not know whether they are strictly locally $\gD$-finite
or in fact locally $\gH$-finite --- only that any of these being
strictly locally $\gD$-finite implies that it cannot be locally finite,
answering in the negative to its Burnside Problem.

Notice that none of the varieties which were found
in Section~\ref{subsec:CR} to be either
strictly locally $\gL$-finite, strictly locally $\gR$-finite
or strictly locally $\gD$-finite belong to $\lat(\mathbf{PCR})$.
Indeed, by Theorem~\ref{T:CRclassification} the only possible candidates
would be varieties $\V$
with either $\mathbf{LRO} \subseteq \V \subseteq \mathbf{ROL^*}$
or $\mathbf{RRO} \subseteq \V \subseteq \mathbf{ROR^*}$
or $\mathbf{RO} \subseteq \V$, respectively.
However, semigroups in $\mathbf{LRO}$ (left regular orthogroups)
need not be periodic;
in fact, left regular orthogroups are semilattices of left groups
and not even a left group needs not be periodic.
(To see this, let $S=L\times G$,
where $L$ is a left zero semigroup and $G$ a group,
be a left group (cf.~Lemma~\ref{L:LG}).
Then, if $x=(a,g) \in S$,
we have $x^k=(a,g^k)\neq x$ unless $G$ is itself periodic.)
Thus, $\mathbf{LRO}$ does not belong to $\lat(\mathbf{PCR})$.
Similarly, nor does $\mathbf{RRO}$.

Moreover, $\lat(\mathbf{PCR})$ and $\lat(\mathbf{Comb})$ contain themselves
an ideal each, respectively, the ideal of all periodic varieties of groups
and the ideal of all \emph{nilvarieties}, that is,
all varieties which satisfy an identity of the form $x^n=0$
for some positive integer $n$.
And if the first of these two consists entirely of locally $\gH$-finite varieties,
in the second we will find locally $\gH$-finite varieties
(as for example the variety $[xy=0]$ of all semigroups with zero multiplication)
along with varieties which are not locally $\gJ$-finite.

\begin{lma}
 The nilvarieties $[x^n=0]$ with $n\geq 2$ are not locally $\gJ$-finite.
\end{lma}

\begin{proof}
 Since $[x^m=0] \subseteq [x^n=0]$ whenever $m\leq n$,
 it suffices to show that $[x^2=0]$ is not locally $\gJ$-finite.
 
 Let $\mathcal{W}$ be the (infinite) set of factors of Morse and Hedlund's
 infinite square-free word on three letters already mentioned in Example~\ref{E:C2}.
 Let $S(\mathcal{W})$ be the $3$-generated infinite semigroup with zero
 whose nonzero elements are the words in $\mathcal{W}$,
 endowed with the operation 
 $u\ast v=uv$ if $uv\in \mathcal{W}$ and $0$ otherwise,
 for all $u,v \in S(\mathcal{W})$ (see~\cite{Sap14}).
 Since $u\gJ v$ implies that $u$ is a factor of $v$ and conversely,
 we have that $J_u=\{u\}$, for every $u \in S(\mathcal{W})$,
 so that $S(\mathcal{W})$ is not $\gJ$-finite.
 As $S(\mathcal{W})$ satisfies the identity $x^2=0$
 and is finitely generated,
 we conclude that $[x^2=0]$ is not locally $\gJ$-finite.
\end{proof}

A similar argument actually shows a stronger result:

\begin{prop}
 If $\V=[u=0]$, then $\V$ is either locally $\gH$-finite
 or not locally $\gJ$-finite.
\end{prop}

\begin{proof}
 If $\V$ is locally finite, then $\V$ is locally $\gH$-finite by
 Remark~\ref{R:finiteness-hierarchy}~(2).
 So suppose $\V$ is not locally finite,
 and let $S$ be an infinite $n$-generated semigroup from $\V$.
 By the universal property of free semigroups,
 we have that $F_n(\V)$, the $n$-generated free semigroup on $\V$, is infinite as well.
 Let $X$ be an alphabet with $n$ letters.
 
 Since we already know that the varieties $[x^n=0]$ are not locally $\gJ$-finite,
 we may assume that $u$ is a word in $m\geq 2$ letters, say, $u_1,\ldots,u_m$,
 of length $l\geq 3$, given that $[xy=0]$ is locally finite;
 write $u=u_{i_1}\ldots u_{i_l}$.
 Then $F_n(\V)=X^+/\tau$,
 where, for all $v,w\in X^+$, we have $(v,w)\in \tau$ if and only if
 $v=w$ or there exist a factor $v'$ of $v$, a factor $w'$ of $w$
 and maps $\pi_v,\pi_w \colon \{u_1,\ldots,u_m\} \to X^+$
 such that $v'=u_{i_1}\pi_v\ldots u_{i_l}\pi_v$ and $w'=u_{i_1}\pi_w\ldots u_{i_l}\pi_w$.
 
 We claim that $F_n(\V)$ is not $\gJ$-finite.
 Let $v\tau,w\tau \in F_n(\V)$ be nonzero elements,
 that is, neither $v$ nor $w$ contain factors of the form $u_{i_1}\pi\ldots u_{i_l}\pi$
 for some mapping $\pi \colon \{u_1,\ldots,u_m\} \to X^+$.
 In particular, neither $v$ nor $w$ are themselves of this form.
 Suppose $(v\tau,w\tau) \in \gJ$.
 Then there exist $p\tau,q\tau,r\tau,s\tau \in (F_n(\V))^1$ such that
 $v\tau=p\tau\,w\tau\,q\tau$ and $w\tau=r\tau\,v\tau\,s\tau$.
 Assuming $v\tau\neq w\tau$, we have that
 $p\tau$ and $q\tau$ cannot be both $1$ (and $r\tau$ and $s\tau$ cannot be both $1$).
 Thus, $v\tau=(pw)\tau$ or $v\tau=(wq)\tau$ or $v\tau=(pwq)\tau$
 and, in either case, the fact that $v\tau$ is a nonzero element in $F_n(\V)$
 implies that $w$ is a factor of $v$.
 Similarly, from $w\tau=r\tau\,v\tau\,s\tau$ we conclude that $v$ is a factor of $w$.
 Therefore, $v=w$, a contradiction.
 Hence, $J_{v\tau}=\{v\tau\}$ for every nonzero $v\tau \in F_n(\V)$.
 As $F_n(\V)$ is infinite, we conclude that it is not $\gJ$-finite.
\end{proof}

Note that the proof of the previous result shows that, in fact,
a variety $[u=0]$ is either locally finite or not locally $\gJ$-finite.


%

\section*{Acknowledgements}

The first author was partially supported by CMUP (UID/MAT/00144/2013),
which is funded by FCT --- Funda\c c\~ao para a Ci\^encia e a Tecnologia (Portugal)
with national (MEC) and European structural funds (FEDER),
under the partnership agreement PT2020.
The second author acknowledges support of the National Funding from FCT,
under the project UID/MAT/4621/2013 (CEMAT-Ci\^encias).


\begin{thebibliography}{99}

\bibitem{Bez01} \textsc{G.~Bezhanishvili},
 \emph{Locally finite varieties},
 Algebra Univers., {\bf 46} (2001), 531--548.

\bibitem{Bropp} \textsc{T.~Brough},
 \emph{Inverse semigroups with rational word problem are finite},
 arXiv: 1311.3955v1 [math.GR].

\bibitem{CliPre67} \textsc{A.~H.~Clifford}, \textsc{G.~B.~Preston}.
 ``The Algebraic Theory of Semigroups'',
 Mathematical Surveys and Monographs, Volume 7, Part 2,
 American Mathematical Society, 1961.


\bibitem{GR52} \textsc{J.~A.~Green}, \textsc{D.~Rees},
 \emph{On semigroups in which $x^r=x$},
 Proc. Cambr. Phil. Soc., {\bf 48} (1952), 35--40.
 
\bibitem{How95} \textsc{J.~M.~Howie}. ``Fundamentals of Semigroup Theory'',
 Clarendon Press, Oxford, 1995.

\bibitem{Kad92} \textsc{J.~Ka\v dourek}.
 \emph{On Varieties of Combinatorial Inverse Semigroups II},
 Semigroup Forum, {\bf 44} (1992), 53--78.



\bibitem{Mal70} \textsc{A.~Malcev},
 ``Algebraicheskie Systemi'' (in Russian),
 Nauka Press, Moscow, 1970.
 English translation: ``Algebraic Systems'',
 Academie-Verlag, Berlin, 1973.

\bibitem{MoHe38} \textsc{M.~Morse}, \textsc{G.~A.~Hedlund},
 \emph{Symbolic dynamics},
 American J. Math., {\bf 60} (1938), 815--866.
 
\bibitem{Mun74} \textsc{W.~D.~Munn},
 \emph{Free inverse semigroups},
 Proc. London Math. Soc.~(3), {\bf 29} (1974), 385--404.
 
\bibitem{Neu67} \textsc{B.~H.~Neumann},
 \emph{Varieties of groups},
 Bull. Amer. Math. Soc., {\bf 73} (1967), 603--613.
 
\bibitem{Pet84} \textsc{M.~Petrich},
 ``Inverse Semigroups'',
 Pure and Applied Mathematics,
 John Wiley \& Sons, 1984.
 
\bibitem{PetRei84} \textsc{M.~Petrich}, \textsc{N.~Reilly},
 \emph{The join of the varieties of strict inverse semigroups and rectangular bands},
 Glasgow Math. J., {\bf 25} (1984), 59--74.
 
\bibitem{PetRei99} \textsc{M.~Petrich}, \textsc{N.~Reilly},
 ``Completely Regular  Semigroups'',
 Canadian Mathematical Society series of monographs and advanced texts
 {\bf 23}, John Wiley \& Sons, 1999.


\bibitem{Sap14} \textsc{M.~V.~Sapir},
 ``Combinatorial Algebra: Syntax and Semantics'',
 Springer Monographs in Mathematics,
 Springer Verlag, 2014.
 

\bibitem{SVV09} \textsc{L.~N.~Shevrin}, \textsc{B.~M.~Vernikov}, \textsc{M.~V.~Volkov},
 \emph{Lattices of Semigroup Varieties},
 Russian Mathematics (Iz.~VUZ), {\bf 50}~No.~3 (2009), 1--28.
 
\bibitem{JBS90} \textsc{J.~B.~Stephen},
 \emph{Presentation of inverse monoids},
 J.~Pure Appl. Algebra, {\bf 63} (1990), 81--112.
 

\end{thebibliography}
\end{document}